\documentclass[notitlepage,11pt,reqno]{amsart}  
\usepackage[foot]{amsaddr}
\usepackage[numbers,sort]{natbib}
\usepackage{amssymb,nicefrac,bm,upgreek,mathtools,verbatim}
\usepackage[final]{hyperref}
\usepackage[mathscr]{eucal}
\usepackage{dsfont} 
\usepackage{graphicx}
\usepackage[margin=10pt,font=small,labelfont=bf]{caption}

\usepackage[margin=0.8in]{geometry} 
\usepackage{babel}
\allowdisplaybreaks  
\usepackage{color}  
\definecolor{dmagenta}{rgb}{.4,.1,.5}       
\definecolor{dblue}{rgb}{.0,.0,.5}     
\definecolor{mblue}{rgb}{.0,.0,.8}     
\definecolor{ddblue}{rgb}{.0,.0,.4}            
\definecolor{dred}{rgb}{.6,.0,.0}   
\definecolor{dgreen}{rgb}{.0,.5,.0}  
\definecolor{Eeom}{rgb}{.0,.0,.5}

\usepackage[normalem]{ulem}

\newtheorem{lemma}{Lemma}[section]
\newtheorem{theorem}{Theorem}[section]
\newtheorem{proposition}{Proposition}[section]
\newtheorem{corollary}{Corollary}[section]

\theoremstyle{definition}

\newtheorem{assumption}{Assumption}[section]

\theoremstyle{remark}
\newtheorem{remark}{Remark}[section]
\numberwithin{equation}{section}

\hypersetup{
  colorlinks=true,
  citecolor=dblue,
  linkcolor=dblue,
  frenchlinks=false,
  pdfborder={0 0 0},
  naturalnames=false, 
  hypertexnames=false,
  breaklinks}
\usepackage[capitalize,nameinlink]{cleveref}

\crefname{section}{Section}{Sections}
\crefname{subsection}{Section}{Sections}
\crefname{condition}{Condition}{Conditions}
\crefname{hypothesis}{Hypothesis}{Conditions}
\crefname{assumption}{Assumption}{Assumptions} 
\crefname{lemma}{Lemma}{Lemmas} 
  
\Crefname{figure}{Figure}{Figures}

\crefformat{equation}{\textup{#2(#1)#3}}
\crefrangeformat{equation}{\textup{#3(#1)#4--#5(#2)#6}}
\crefmultiformat{equation}{\textup{#2(#1)#3}}{ and \textup{#2(#1)#3}}
{, \textup{#2(#1)#3}}{, and \textup{#2(#1)#3}}
\crefrangemultiformat{equation}{\textup{#3(#1)#4--#5(#2)#6}}%
{ and \textup{#3(#1)#4--#5(#2)#6}}{, \textup{#3(#1)#4--#5(#2)#6}}%
{, and \textup{#3(#1)#4--#5(#2)#6}}

\Crefformat{equation}{#2Equation~\textup{(#1)}#3}
\Crefrangeformat{equation}{Equations~\textup{#3(#1)#4--#5(#2)#6}}
\Crefmultiformat{equation}{Equations~\textup{#2(#1)#3}}{ and \textup{#2(#1)#3}}
{, \textup{#2(#1)#3}}{, and \textup{#2(#1)#3}}
\Crefrangemultiformat{equation}{Equations~\textup{#3(#1)#4--#5(#2)#6}}%
{ and \textup{#3(#1)#4--#5(#2)#6}}{, \textup{#3(#1)#4--#5(#2)#6}}%
{, and \textup{#3(#1)#4--#5(#2)#6}}

\crefdefaultlabelformat{#2\textup{#1}#3}
\newcommand{\cX}{{\mathcal{X}}}  

\newcommand{\veps}{\varepsilon}

\newcommand{\beql}[1]{\begin{equation}\label{#1}}

\newcommand{\beq}{\begin{displaymath}}
\newcommand{\eeqno}{\end{displaymath}}
\newcommand{\eeq}{\end{equation}}

\newcommand{\E}{\mathbb{E}}
\newcommand{\PP}{\mathbb{P}}

\newcommand{\RR}{\mathds{R}}
\newcommand{\NN}{\mathds{N}}

\newcommand{\Ind}{\mathds{1}}   

\newcommand{\transp}{^{\mathsf{T}}}

\DeclareMathOperator*{\esssup}{ess\,sup}

\newcommand{\grad}{\nabla}

\newcommand{\calP}{\mathcal{P}}

\newcommand{\cC}{\mathcal{C}}
\newcommand{\cV}{\mathscr{V}}

\newcommand{\calS}{\mathcal{S}}

\newcommand{\calC}{\mathcal{C}}

\newcommand{\frL}{\mathfrak{L}}

\newcommand{\cZ}{\mathcal{Z}}

\newcommand{\calF}{\mathcal{F}}

\makeatletter
\DeclareRobustCommand\widecheck[1]{{\mathpalette\@widecheck{#1}}}
\def\@widecheck#1#2{%
    \setbox\z@\hbox{\m@th$#1#2$}%
    \setbox\tw@\hbox{\m@th$#1%
       \widehat{%
          \vrule\@width\z@\@height\ht\z@
          \vrule\@height\z@\@width\wd\z@}$}%
    \dp\tw@-\ht\z@
    \@tempdima\ht\z@ \advance\@tempdima2\ht\tw@ \divide\@tempdima\thr@@
    \setbox\tw@\hbox{%
       \raise\@tempdima\hbox{\scalebox{1}[-1]{\lower\@tempdima\box
\tw@}}}%
    {\ooalign{\box\tw@ \cr \box\z@}}}
\makeatother

\usepackage{accents}
\newlength{\dhatheight}

\newcommand{\D}{ \mathrm{d}}

\let\oldtocsection=\tocsection
\let\oldtocsubsection=\tocsubsection
\let\oldtocsubsubsection=\tocsubsubsection
\renewcommand{\tocsection}[2]{\hspace{0em}\oldtocsection{#1}{#2}}
\renewcommand{\tocsubsection}[2]{\hspace{1em}\oldtocsubsection{#1}{#2}}
\renewcommand{\tocsubsubsection}[2]{\hspace{2em}\oldtocsubsubsection{#1}{#2}}

\newcommand{\ttl}{\Large Strong and weak quantitative estimates in slow-fast diffusions using filtering techniques}

\newcommand{\ttls}{Convergence estimates in slow-fast diffusions}

\begin{document}

\title[\ttls]{\ttl}

\author{Sumith Reddy Anugu$^\dag$}
\author{Vivek S. Borkar$^\ddag$}
\address{$^\dag$Institute f\"ur Mathematik, Technische Universit\"at Ilmenau,
    Weimarer~Str.~25, 98693 Ilmenau, Germany} \email{sumith-reddy.anugu@tu-ilmenau.de}
\address{$^\ddag$Department of Electrical Engineering,\\
 Indian Institute of Technology Bombay, Powai, Mumbai, India 400076. Work supported by a grant from Google Research Asia.}
\email{borkar.vs@gmail.com}
	
	\begin{abstract}
The behavior of slow-fast diffusions as the separation of scale diverges  is a well-studied problem in the literature.  In this short paper, we revisit this problem  and  obtain a new  proof of existing    strong quantitative convergence  estimates (in particular, $L^2$ estimates) and weak convergence estimates    in terms of $n$ (the parameter associated with the separation of scales). In particular, we obtain the rate of $n^{-\frac{1}{2}}$ in the strong convergence estimates and the rate of $n^{-1}$ for weak convergence estimate which are already known to be optimal in the literature.  We achieve this using nonlinear filtering theory  where we represent the evolution of  fast diffusion in terms of  its conditional distribution  given the  slow diffusion. We then use the well-known  Kushner-Stratanovich equation which gives the evolution of the conditional distribution of the fast diffusion given the slow diffusion and establish that this conditional distribution  approaches the invariant measure of the ``frozen" diffusion (obtained by freezing the slow variable in the evolution equation of the fast diffusion). At the heart of the analysis lies a key  estimate of a weighted {Lipschitz distance like function between a generic one-parameter family of measures and the family of  unique invariant measures (of the ``frozen" diffusion parametrized by a path)}. This estimate is in terms of the operator norm of the dual of the infinitesimal generator of the ``frozen" diffusion.
	\end{abstract}
	\maketitle
	\section{Introduction} 
	Multi-scale phenomena are ubiquitous in many areas of science and engineering such as neurosciences \cite{Rinzel1987, izhikevich2007dynamical,
 amir2002burst}, weather forecasting \cite{powers2017weather}, cellular dynamics \cite{harvey2011multiple}, chemical systems \cite{krischer1992oscillatory}, ecology \cite{rinaldi2000}, etc (see~\cite{BERTRAM2017, kristiansen2023review,perryman2014adapting} for elaborate discussions on applications). 
	  In this paper, we consider the following  diffusion given as a system of stochastic differential equations (SDEs): 
	\begin{align}\label{eq-intro-X}
\D X^n_t&=b(X^n_t,Y^n_t)\D t +\sigma(X^n_t) \D W_t,\,&X^n_0= \bar x,\\\label{eq-intro-Y}
\D Y^n_t&=nh(X^n_t,Y^n_t)\D t +\sqrt n \eta(X^n_t,Y^n_t) \D B_t,\,& Y^n_0=\bar y\,.
\end{align}
We are interested in the behavior of the above diffusion as $n$ becomes very large. It is clear that for large $n$, the process $Y^n$ evolves much faster than the process $X^n$ - for this reason, we refer to $(X^n,Y^n)$  as slow-fast diffusion. Over the time scale of $\mathscr{O}(1)$, if $X^n=X_1$ in the beginning of the timescale, the evolution of $Y^n$ over timescale $\mathscr{O}(n^{-1})$ can be approximated well by treating $X^n\simeq X_1$ (throughout the time interval of the order $n^{-1}$) and this approximation of $Y^n$ can be used to obtain the evolution of $X^n$ over the time increments of order $n^{-1}$. In other words, we consider the following SDE (often referred to as ``frozen" diffusion) 
\begin{align*} \D Y^{*,z}_t&=h(z,Y^{*,z}_t)\D t + \eta(z,Y^{*,z}_t) \D B_t,\,& Y^{*,z}_0=\bar y\end{align*} with $z$ as a parameter.  Under an appropriate  ``uniform in $z$" stability condition (see Assumption~\ref{assump-stab}) on $Y^{*,z}$ (with $\pi^{*,z}$ as the unique invariant measure), it is  shown in the literature (see below for the literature survey) that the process $X^n$ approaches $\bar X$ which is given as a solution to 
\begin{align}\label{eq-intro-lim}\D \bar X_t&=\int b(\bar X_t,y)\pi^{*,\bar X_t}(\D y)\D t +\sigma(\bar X_t) \D W_t,\,& \bar X_0= \bar x\,.\end{align}
Such a result is often referred to as the averaging principle. 

\subsection*{Literature survey} In the following, we give a survey of existing literature (which is by no means complete) on the problem described above. This problem was first studied in the context of ordinary differential equations  by N. Krylov and N. Bogolyubov in \cite{krylov1935methodes} and in the context of SDEs by R. Z. Khasminskii in \cite{khasminskij1968principle}. Since then the averaging principle was investigated extensively under various settings. For an early comprehensive treatment on this topic, we refer the reader to \cite{freidlin2012random}. The large $n$ behavior had also been investigated in the context of stochastic control and filtering problems; see \cite{kushner2012weak,bardi2023singular, kouhkouh2022some,hicham2024,beeson2022approximation,qiao2023convergence,qiao2024limit,bardi2024deep,borkar2007singular}. In literature, other variations of slow-fast diffusion with drastically different limiting behavior were also studied; see \cite{kabanov2003two} for a recent exposition in this regard.

We  discuss the relevant literature below which we remark again  is, unsurprisingly, not exhaustive.
In the last two decades, this problem received much attention and was investigated in other contexts like the distribution dependent SDEs (or the Mckean-Vlasov type SDEs; for example, see \cite{bezemek2023interacting, cheng2024, hao2023mckean, HONG2022,rockner2021strong,xu2021strong}) or  the case of L\'evy process driven diffusions (for example, see \cite{dror2007,sun2023,xiaobin2022,zhang2018}). Below,  we focus our attention mainly on the existing works in the context of  SDEs driven by Brownian motion due to its relevance to the work in this paper.  The existing results can be classified into two types depending how $X^n$ and $Y^n$ are coupled.
\begin{enumerate}
\item (Partially coupled) In this case, $Y^n$ is coupled to  $X^n$ only through drift coefficient, \emph{i.e.,} $\sigma(x,y)$ is independent of $y$.  The strong convergence  estimate of $\|X^n_t-\bar X_t\|^2$ under H\"older continuity of the coefficients  is investigated in \cite{rockner2019strongweakconvergenceaveraging,rockner2021averaging}. In  \cite{feo2023, de2023sdes, ge2024}, this problem is considered in the case of SDEs on Hilbert spaces. Using an asymptotic expansion approach, authors in \cite{khasminskii2004} established the averaging principle. In  \cite{weinan2005analysis},  the convergence of numerical schemes of the slow-fast dynamical system (with fast dynamics being a diffusion and slow dynamics being an ODE) is studied. The case of stochastic functional differential equations is considered in \cite{wu2022fast,wu2020averaging}. In addition to this, estimates on weak convergence of $X^n$ to $\bar X$ are also studied in  \cite{rockner2019strongweakconvergenceaveraging} under H\"older continuity of coefficients.

\item (Fully coupled) In this case, $\sigma(x,y)$ is allowed to  depend on $y$ and the limiting process $\bar X$ (if it exists) evolves according to 
$$\D \bar X_t=\int b(\bar X_t,y)\pi^{*,\bar X_t}(\D y)\D t +\bar \sigma(\bar X_t) \D W_t,\,\quad \bar X_0= \bar x$$
with $\bar \sigma(\cdot)$ such that $\bar \sigma(x)\bar \sigma(x)\transp = \int\sigma(x,y)\sigma(x,y)\transp \pi^{*,x}(\D y)$. 
  This case is much less studied and the works in this direction include \cite{liu2010,  rockner2021diffusion,brehier2022} where only weak convergence of $X^n $ to $\bar X$ is established. The important difference between the fully coupled and the partially coupled case is that the strong convergence estimates fail to hold in general. We refer the reader to \cite[Section 4]{liu2010} for an illustration of this fact.

\end{enumerate}

In comparison with the above literature, our work in this paper differs significantly in its  starting point - much of the literature uses analytical techniques and/or works with~\eqref{eq-intro-X}--\eqref{eq-intro-Y}.  The starting point of our approach uses the tools from nonlinear filtering theory, in particular, the Kushner-Stratanovich equation which is the evolution equation of the conditional distribution of $Y^n$ given the filtration generated by $X^n$ (say, $\pi^n_t(\cdot)$). As a consequence, we re-express the evolution equation of the slow process $X^n$ completely in terms of $X^n$ and $\pi^n_t(\cdot)$. More precisely, $X^n$ now evolves according to 
\begin{align}\label{eq-intro-X-1} \D X^n_t=\int b(X^n_t,y)\pi^n_t(\D y)\D t +\sigma(X^n_t) \D \widetilde W^n_t,\,\quad X^n_0= \bar x,\end{align}
for some Brownian motion $\widetilde W^n$ adapted to the filtration of $X^n$.
Comparing~\eqref{eq-intro-X} with the above display, we see that the drift $b(X^n,Y^n)$ is replaced by   $\int b(X^n_t,y)\pi^n_t(\D y)$ and the Brownian motion $W$ is replaced by $\widetilde W^n$.  The big advantage of this approach is that both  $\int b(X^n_t,y)\pi^n_t(\D y)$ (the drift in~\eqref{eq-intro-X-1}) and   $\int b(\bar X_t,y)\pi^{*,\bar X_t}(\D y)$ (the drift in~\eqref{eq-intro-lim}) are integrals with respect to probability measures $\pi^n_t$ and $\pi^{*,\bar X_t}$, respectively. This means that  the difference between two drifts mentioned above, can be bounded by an appropriate `distance' between $\pi^n_t$ and $\pi^{*,\bar X_t}$. It turns out that this `distance' resembles the weighted {Lipschitz} metric. Under the assumption of exponential ergodicity of $Y^{*,z}$, uniform in $z$,  the associated  Lyapunov function is taken as the weight. From here, with the help of results from the ergodic theory of diffusions, we estimate this weighted {`distance'  between $\{\pi^n_t:0\leq t\leq T\}$ and $\{\pi^{*,\bar X_t}:0\leq t\leq T\}$} under the assumptions of {appropriate Lipschitz continuity of $b,h,\sigma,\eta$, differentiability of $\eta$ (in fast variable)  and linear growth  of the coefficients $h,\sigma$ and the boundedness of coefficients $b$  and $\eta$. This consequently helps us prove both strong and weak convergence estimates. This is in contrast to the existing results in the literature where boundedness and certain smoothness  of coefficients are assumed;  see for example, the regularity conditions in the hypothesis of \cite[Theorem 2.1]{rockner2019strongweakconvergenceaveraging} (for $\alpha=1$ in that paper). On the other hand, our approach in this paper requires a more quantitative version of stability (hence, stronger form of stability), whereas the existing literature typically assumes a more qualitative form of stability; for illustrative purpose, compare  Condition (Hb) of \cite{rockner2019strongweakconvergenceaveraging}  with Assumption~\ref{assump-stab}. Moreover, we are able to prove weak convergence estimate for test functions that are twice differentiable (with H\"older continuous second  derivatives) which is a larger class than the class used in \cite[Theorem 2.5]{rockner2019strongweakconvergenceaveraging} for $\alpha=1$ \emph{viz.} the set of thrice differentiable functions.}

Even though our approach allows for  weaker conditions on coefficients, there is one drawback to our approach: since $\widetilde W^n$ in~\eqref{eq-intro-X-1} a priori depends on $n$, our analysis only concludes that $X^n-X^{*,n}$ approaches  zero, where $X^{*,n}$   is the solution to 
\begin{align}\label{eq-intro-X-2} \D X^{*,n}_t=\int b(X^{*,n}_t,y)\pi^{*,X^{*,n}_t}(\D y)\D t +\sigma(X^{*,n}_t) \D \widetilde W^n_t,\,\quad X^{*,n}_0= \bar x\,.\end{align}

In particular, we show that under our assumptions on coefficients and exponential ergodicity of $Y^{*,z}$,  $X^n-X^{*,n}$ converges to zero in $L^2$ at rate of $n^{-\frac{1}{2}}$ and also show that $X^n$ converges weakly to $\bar X$ (note that $\bar X$ and $X^{*,n}$ have the same law) at a rate $n^{-1}$. It is now well known from the work of \cite{liu2010} that these rates are optimal.

	This short paper is organized as follows. We briefly introduce the notation used in the paper in the rest of this section. In Section~\ref{sec-model-results}, we give the description of our slow-fast diffusion model and recall a useful representation from \cite{wong1971} in the context of  the slow variable.  The statements of our main results  also given in this  section. In Section~\ref{sec-filter}, we  prove key lemma involving conditional distribution of fast variable (given slow variable) and the invariant measure of the ``frozen" diffusion.    Finally, we provide the proofs of our main results in Section~\ref{sec-proof}.
	
	\subsection*{Notation} $(\Omega,\calF, \calF_t,\PP)$ denotes our abstract complete filtered probability space. For a process $Z$, $\{\calF_t(Z):t\geq 0\}$ denotes the filtration generated by it. For any $k\in \NN$, $\|x\|$ denotes the Euclidean norm of $x\in \RR^k$ {and $\langle x,y\rangle$ denotes the Euclidean inner product on $\RR^k$}. For a  probability measure $\mu$ on a measurable space $S$ with $\sigma$-algebra $\calS$ and an $\calS$-measurable function $f:S\rightarrow \RR$,  $\mu[f]\doteq\int_{S}f(x)\mu(\D x)$, $\calP(S)$ denotes the set of all probability measures on $(S,\calS)$ and $\|f\|_\infty=\esssup_{x\in S} |f(x)|$. {If $f: S\rightarrow \RR^k$, with a slight abuse of notation, $\|f\|_\infty= \max_{1\leq i\leq k}\esssup_{x\in S} |f_i(x)|$.} 
$\Ind_{\{ \ \cdot \ \in \ A\}}$ denotes the indicator function {corresponding to}  Borel set $A\subset \RR^k$. { $B_R(z)$ denotes the open ball of radius $R$ around $z\in \RR^k$. For $k,m\in \NN$ and $0<\theta<1$, 
	$\calC^{m,\theta}(\RR^k)$ denotes the space of functions on $\RR^k$ that are $m$-times differentiable with $m$th derivative being $\theta$-H{\"o}lder continuous. The set $\calC_b^{m,\theta}(\RR^k) $ denotes the set of $u\in \calC^{m,\theta}(\RR^k)$ such that $\|u\|_{m,\theta,\infty}<\infty$. Here,
	$$ \| u\|_{m,\theta,\infty}\doteq \|u\|_\infty + \sum_{|\beta|=0}^m\|D^\beta u\|_\infty + \sup_{x\neq y\in \RR^k} \frac{|u(x)-u(y)|}{|x-y|^\theta},$$
	with $D^\beta u$ being  the $\beta$th order derivative of $u$. $\calC^{m}(\RR^k)$ denotes the set of $m$-times differentiable functions with continuous derivatives up to $m$th order. Analogously, $\calC_b^{m}(\RR^k)$ is defined. Finally, for $k_1,k_2\in \NN$, a map $f:\RR^{k_1}\to \RR^{k_2}$ is said to be Lipschitz, if there exists a constant $K>0$ such that for every $x,y\in\RR^{k_1}$, 
	$$ \|f(x)-f(y)\|\leq K\|x-y\|\,.$$
	The smallest $K$ such that the above inequality holds is referred to as the Lipschitz constant (denoted by $\text{Lip}(f)$). However, any such constant $K$ satisfying the above inequality is also denoted by $\text{Lip}(f)$, with some abuse of notation. We use $C$ to denote a  generic positive constant whose finiteness is more important than the actual value. 
	}
	
	\section{Model and results} \label{sec-model-results}
Let $(W,B)$ be a pair of  independent $p$ and $q$--dimensional Brownian motions. We consider a family $\{(X^n,Y^n):n\in \NN\}$ of $\calC(\RR_+, \RR^p\times \RR^q)$--valued random variables with $\calC(\RR_+, \RR^p\times \RR^q)$ being the space of $\RR^p\times \RR^q$--valued continuous functions on $\RR_+$ equipped with uniform topology on compact sets. The pair  $(X^n,Y^n)$ is a $(p+q)$--dimensional slow-fast diffusion given as a solution to
\begin{align}\label{eq-X}
\D X^n_t&=b(X^n_t,Y^n_t)\D t +\sigma(X^n_t) \D W_t,\,& X^n_0= \bar x,\\\label{eq-Y}
\D Y^n_t&=nh(X^n_t,Y^n_t)\D t +\sqrt n \eta(X^n_t,Y^n_t) \D B_t,\,& Y^n_0=\bar y\,.
\end{align}	
Here, $b:\RR^p\times \RR^q \rightarrow \RR^p$, $h:\RR^p\times \RR^q\rightarrow \RR^q$, $\sigma:\RR^p\rightarrow \RR^{p\times p}$ and $\eta:\RR^p\times \RR^q\rightarrow \RR^{q\times q}$. Clearly for large $n$, $Y^n$ is the fast diffusion  and $X^n$ is the slow diffusion.
\begin{remark} Our approach in this paper also allows for $(W,B)$ to be dependent with say, a cross-quadratic variation $\langle W,B\rangle_t\not\equiv 0$ (that does not dependent of $n$). In this case, the second integral in~\eqref{eq-KS} below which is 
$$ \int_0^t \Big(\pi^n_{r}[f(X^n_r,\cdot)b^\dagger(X^n_r,\cdot)]-\pi_{r}^n[f(X^n_r,\cdot))]\pi^n_{r}[b^\dagger(X^n_r,\cdot)]\Big)\sigma^{-1}(X^n_r)\D I^n_r$$ will be replaced by 
$$\int_0^t \Big(\pi^n_r[Q_r]+\Big(\pi^n_{r}[f(X^n_r,\cdot)b^\dagger(X^n_r,\cdot)]-\pi_{r}^n[f(X^n_r,\cdot))]\pi^n_{r}[b^\dagger(X^n_r,\cdot)]\Big)\sigma^{-1}(X^n_r)\Big)\D I^n_r$$ 
with $Q_t$ given by 
$Q_t\doteq \frac{\D\langle W,B\rangle_t }{\D t}$. It can be seen from our approach that follows, that this case can be handled in the same way as the case where $W$ and $B$ are independent.  For sake of keeping expressions in the  analysis  simple, we assume that $W$ and $B$ are independent.
\end{remark}

To proceed further, we impose certain   regularity and ellipticity conditions on the coefficients.
{\begin{assumption}\label{assump-regularity} The coefficients $b$, $\sigma$, $h$ and $\eta$ satisfy the following:
\begin{enumerate} 
\item [(i)]  There exists $L>0$ and $y_0\in\RR^q$ such that for $x\in \RR^p$ and $y\in \RR^q$, 
\begin{align*}\|h(z,y_0)\|&\leq L,\quad\quad\quad \|\eta(x,y)\|+\|b(x,y)\|\leq L,\\
\sup_{x\in \RR^p} \|h(x,y)\|^2&\leq L(1+\|y\|^2),\quad\text{and}\quad \|\sigma(x)\|^2\leq L (1+ \|y\|^2)\,.\end{align*}
\item [(ii)]The maps  $b,h,\sigma$ and $\eta$  are  Lipschitz. Moreover,  for every $z\in \RR^p$,  $\eta(z,\cdot)$ is  differentiable  and $\nabla_y\eta(\cdot,y)$ is Lipschitz with Lipschitz constant uniformly bounded in $y$.  
\item [(iii)] There exists $\delta>0$ such that  for every $x,z\in \RR^p$ and $y,w\in \RR^q$,
$$ z^\dagger \sigma(x)\sigma^\dagger (x)z\geq \delta\|z\|^2 \quad \text{and}\quad  w^\dagger \eta(x,y)\eta^\dagger (x,y)w\geq \delta  \|w\|^2\,.$$
\end{enumerate}
\end{assumption}
}
\begin{remark}
The reader may notice that the diffusion coefficient $\sigma$ in~\eqref{eq-X} is independent of $Y^n$ which is unlike the diffusion coefficient $\eta$ in~\eqref{eq-Y}  that depends on both $X^n$ and $Y^n$. With regards to the approach used in this paper, the reason behind this choice in this paper is a technical one - as mentioned already, we make a heavy use of the representation theorem due to E. Wong in \cite{wong1971} and it is not clear if the representation still remains valid if we introduce the dependence of $Y^n$ into the diffusion coefficient $\sigma$ in the equation for $X^n$. As mentioned in the introduction,  most of the  existing works on slow-fast diffusions allow the diffusion coefficient of the fast diffusion to depend only on fast variable. In \cite{rockner2021diffusion}, authors investigate this problem in the context of SDEs while allowing for $\sigma$ to depend on the fast variable $Y^n$.   \end{remark}
{\begin{remark} We give a few comments on Assumption~\ref{assump-regularity}. 
The first inequality in (i) is imposed in order to invoke \cite[Theorem 1.2]{menozzi2021} which is a crucial result that is used to establish Lemma~\ref{lem-property}. The boundedness of $b$ (the second inequality in (i))  is used crucially only in proving Lemma~\ref{lem-lipschitz}; whereas, the boundedness of $\eta$ and aforementioned differentiability of $\eta$ is used in proof of Lemma~\ref{lem-lipschitz-inv-meas} to invoke an existing result from \cite[Theorem 2.2]{bogachev2018}.
\end{remark}}

To analyse the limiting behavior as $n\to \infty$, we consider an intermediate ``frozen" process $Y^{n,z}$ (with both $n\in \NN$ and $z\in \RR^p$  as parameters) given as a strong solution to 
\begin{align}\label{eq-yforfixedx}
\D Y^{n,z}_t=n h(z,Y^{n,z}_t)\D t +{\sqrt{n}}\eta(z,Y^{n,z}_t) \D B_t,\quad Y^{n,z}_0=y\,.
\end{align}
Again, for every $z\in \RR^p$, from the conditions on  $h$ and $\eta$ in Assumption~\ref{assump-regularity}, we know that the process $Y^{n,z}$  exists uniquely in the strong sense. {For $(z,y)\in \RR^p\times \RR^q$, let $\frL^z:\calC^2(\RR^q)\rightarrow \calC(\RR^q)$ and $\widehat \frL^z:\calC^2(\RR^p)\rightarrow \calC(\RR^p)$  be operators defined as
\begin{align}
\frL^z[f](y)&\doteq \text{Tr}\big(\eta(z,y )\eta^\dagger(z,y) D_y^2 f(y)\big) + h(z,y )\cdot \grad_y f(y)\\
 \widehat {\mathfrak L}^{z,y} [f](z)&\doteq  \text{Tr}\big(\sigma(z )\sigma^\dagger(z) D_x^2 f(z)\big) + b(z,y )\cdot \grad_x f(z)\,.
\end{align}}
It is clear that the operator $n\mathfrak{L}^z$ is the infinitesimal generator of the process $Y^{n,z}$, for every $n\in \NN$ and $z\in \RR^p$. { With a slight abuse of notation, we write $Y^{1,z}$ as $Y^z$.} 

The following    stability assumption  on the ``frozen" process $Y^{n,z}$  is imposed to ensure the existence of a non-trivial limiting behavior of $\{X^n:n\in \NN\}$ as $n\to \infty$.
\begin{assumption}\label{assump-stab} There exists a symmetric positive definite $q\times q$ matrix $A$  and constants $\delta_0,\delta_1>0$ such that   for $z\in \RR^p$,
\begin{align}\label{eq-exp-stab-assump}
 \langle h(z,y),A y\rangle  \leq \delta_0 - \delta_1\langle y,Ay\rangle\,.
\end{align}
\end{assumption}
{
\begin{remark}\label{rem-k-moment}
Define,  $\cV_k(y)= (\langle y,Ay\rangle)^k$. It can be easily verified that for any $k\geq 1$, there are constants $\beta_0(k),\beta_1(k)>0$ such that for $z\in \RR^p$ and $y\in \RR^q$, 
\begin{align*}
\mathfrak L^z[\cV_k](y)\leq \beta_0(k)-\beta_1(k) \cV_k(y)\,.
\end{align*}
From the above display, an application of It\^o's formula to $\cV_k(Y^n_t)$ gives us 
$$ \sup_{0\leq t\leq T}\E\big[\cV_k(Y^n_t)\big]\leq \cV_k(\bar y) + \beta_0(k) T\,.$$
\end{remark}
In the rest of the paper, we simply set $\cV_1=\cV$, $\beta_0(1)=\beta_0$  and $\beta_1(1)=\beta_1$.  } Set $\cZ\doteq \cC^2(\RR^q)\cap \cC_\cV(\RR^q) $, where  $\cC_\cV(\RR^q)$ as the set of continuous functions $f:\RR^q\rightarrow \RR$ such that 
\begin{align*}
\limsup_{\|y\|\to\infty}\frac{|f(y)|}{\cV(y)}<\infty\,.
\end{align*}
\begin{remark}\label{rem-inv-meas}
It is well known that  under the above assumption, for every $z\in \RR^p$, the diffusion {$Y^{z}$} is exponentially ergodic (uniformly in $z\in \RR^p$) and in particular, has a unique invariant measure which we denote as  $\pi^{*,z}$; see \cite[Theorem 2.6.10]{arapostathis2011}. Moreover, from \cite[Theorem 2.6.16]{arapostathis2011}, it is known that $\pi\in \calP(\RR^q)$ is the unique invariant measure of {$Y^{z}$} if and only if $ \pi\big[\mathfrak L^z[g]\big]=0$, for every $g\in \cZ$. Also, we have {$\sup_{z\in \RR^p} \pi^{*,z}[\cV]<\frac{\beta_0}{\beta_1}$.}
\end{remark}
We now re-express the process $X^n$ in a form that is ``entirely in terms of itself". To do this, we make use of the notion of conditional distribution. This is the foundation of our approach in this paper.  More precisely, it involves giving an alternative evolution equation for $X^n$, where the drift and the diffusion terms are adapted with respect to $\calF_t(X^n)$, instead of $\calF_t(W,B)$. In particular, the drift term involves the conditional distribution of fast variable $Y^n$ given $\calF_t(X^n)$. We then recall the well-known Kushner-Stratanovich equation (see~\eqref{eq-KS} below) that dictates the evolution of the  conditional distribution of $Y^n$ given $\calF_t(X^n)$.    Let $\pi^n_t(\cdot)$ be the conditional distribution of $Y^n_t$ given $\calF_t(X^n)$, \emph{i.e.,} $$\pi^n_t(\cdot)\doteq \PP(Y^n_t\in \cdot \ |\calF_t(X^n))\,.$$ It is well-known that for $f$ in an appropriate function class, $\pi^n_t[f(X^n_t,\cdot)]$ is an $\calF_t(X^n)$--semi-martingale. Moreover, we have the following well-known result (\cite[Theorem 8.1]{liptser2001}).
\begin{proposition}\label{prop-KS} {Suppose $f\in \cC^2(\RR^p\times \RR^q)$ such that 
\begin{align}\label{eq-KS-cond}
\sup_{0\leq t\leq T}\E\big[\big(f(X^n_t,Y^n_t)\big)^2\big]<\infty\quad \text{and}\quad \int_0^T \E\Big[\Big(\big|\widehat {\mathfrak L}^{X^n_r,Y^n_r}[f(X^n_r,Y^n_r)\big|^2+ \big|\mathfrak L^{X^n_r} [f(X^n_r,Y^n_r)]\big|^2\Big)\D r\Big]<\infty\,.
\end{align}
Then the following holds: for $0\leq t\leq T$ }
\begin{align}\nonumber
\pi^n_{t}[f(X^n_t,\cdot)]&=f(\bar x,\bar y)+{\int_0^t \pi_{r}^n\big[\widehat {\mathfrak L}^{X^n_r,\cdot}[f(X^n_r,\cdot)]+ \mathfrak L^{X^n_r} [f(X^n_r,\cdot)]\big]\D r}\\\label{eq-KS}
 &\quad + \int_0^t \left(\pi^n_{r}[f(X^n_r,\cdot)b^\dagger(X^n_r,\cdot)]-\pi_{r}^n[f(X^n_r,\cdot))]\pi^n_{r}[b^\dagger(X^n_r,\cdot)]\right)\sigma^{-1}(X^n_r)\D I^n_r, \quad \text{$\PP$--a.s.}
\end{align}
Here, $I^n_t\doteq  \int_0^t \sigma^{-1}(X^n_r)\big(\D X^n_r-\pi^n_{r}[b(X^n_r,\cdot)]\D r\big) $.

\end{proposition}
\begin{remark}\label{rem-innovation}
The process $I^n$ is referred to as the innovations process and is a $p$--dimensional  $\calF_t(X^n)$--adapted Brownian motion.  Under certain conditions (see \cite[Theorem 1]{allinger1981new}), $\calF_t(I^n)$ is  identical to $\calF_t(X^n)$ modulo $\PP$--null sets, for every $t\geq 0$. 
\end{remark}

In the next result, as mentioned above we  re-express the evolution of the process $X^n$ completely in  terms of $\calF_t(X^n)$-adapted processes. This is the cornerstone for our approach in this paper.
\begin{proposition}For $n\in \NN$, there exists a $p$--dimensional $\calF_t(X^n)$--adapted Brownian motion $\widetilde W^n$ such that   in terms of $\widetilde W^n$, the process $X^n$ given by~\eqref{eq-X} can be alternatively represented  in the following form: for $0\leq t\leq T$,
\begin{align}\label{eq-representation}
 X^n_t&= \bar x+ \int_0^t\int_{\RR^q} b(X^n_r,y)\pi^n_r(\D y)\D r +\int_0^t\sigma(X^n_r) \D \widetilde W^n_r,\quad \text{$\PP$--a.s.}
\end{align} 
Moreover, $\widetilde W^n=I^n$.
\end{proposition} 
\begin{proof} The first part of the proposition can be proved in two ways: \emph{via.} 1.) direct application of a very general result involving semi-martingales (which is Theorem 4.3 of \cite{wong1971}) and 2.) immediate consequence of the definition of $I^n$ (from Proposition~\ref{prop-KS}) and  the fact that $I^n$  is an $\calF_t(X^n)$--adapted Brownian motion.   Now, the second part follows from the definition of $I^n$.
\end{proof}

\begin{remark} Note that $\widetilde W^n=I^n$ is a different Brownian motion (compared to $W$) which is adapted to a smaller filtration  $\calF_t(X^n)$ instead of $\calF_t(W,B)$. The advantage of the above representation is that the dependence of $Y^n$ is only through $\pi^n_t(\D y)$. 
\end{remark}

To summarize, in the above we have shown that instead of using the evolution  of  the pair $(X^n,Y^n)$ (given by~\eqref{eq-X}--\eqref{eq-Y}) to infer the evolution of  $X^n$, we can use the evolution of pair $(X^n,\pi^n)$ which is given as follows: for every $f\in \cC^2(\RR^p\times \RR^q)$ that satisfies~\eqref{eq-KS-cond},
\begin{align}\nonumber
\D X^n_t&=\int_{\RR^q} b(X^n_t,y)\pi^n_t(\D y)\D t +\sigma(X^n_t) \D I^n_t,\quad X^n_0=\bar x, \quad \text{$\PP$--a.s.}\\\nonumber
\pi^n_{t}[f(X^n_t,\cdot)]&=f(\bar x,\bar y)+{n}\int_0^t \pi_{r}^n\big[\mathfrak L^{X^n_r} [f(X^n_r,\cdot)]\big]\D r\\\nonumber
 &\quad + \int_0^t \left(\pi^n_{r}[f(X^n_r,\cdot)b^\dagger(X^n_r,\cdot)]-\pi_{r}^n[f(X^n_r,\cdot))]\pi^n_{r}[b^\dagger(X^n_r,\cdot)]\right)\D I^n_r, \quad \text{$\PP$--a.s.}
\end{align}

We are now in a position to state our main results of our paper.
\begin{theorem}\label{thm-main}[Strong convergence estimate] Suppose Assumptions~\ref{assump-regularity} and~\ref{assump-stab} hold. Then,  the equation  
\begin{align}\label{eq-limit}
\D   X^{*,n}_t= \bar b( X^{*,n}_t)\D t + \sigma( X^{*,n}_t) \D  I^n_t, \, X^{*,n}_0=\bar x
\end{align} 
with $ \bar b(x)= \int_{\RR^q}b(x,y)\pi^{*,x}(\D y)$ has a unique strong solution. Moreover, the following hold: for every $T>0$ and $1<m<2$,
\begin{align}\label{eq-thm-est-1}
\sup_{0\leq t\leq T}\E\Big[\| X^n_t- X^{*,n}_t\|^2\Big]&\leq C n^{-1}\\\label{eq-thm-est-2}
\E\Big[\sup_{0\leq t\leq T}\| X^n_t- X^{*,n}_t\|^m\Big]&\leq C n^{-\frac{m}{2}},
\end{align}
for some $C=C(T,m)>0$.
\end{theorem}
\begin{remark}\label{rem-sup-est}{Since $b$ and $\bar b$ are  bounded and Lipschitz (due to Assumption~\ref{assump-regularity}(i)-(ii) and  Lemma~\ref{lem-lipschitz}, respectively), $\sigma$ is Lipschitz with linear growth (due to Assumption~\ref{assump-regularity}(i) and (iii)), it is well known (see~\cite[Theorem 5.2.9]{karatzas1991brownian} for instance) that processes $X^n$ and $X^{*,n}$ which are solutions of~\eqref{eq-X}--\eqref{eq-Y} and~\eqref{eq-limit}, respectively are such that 
\begin{align}\label{eq-moment-est} \E\big[\sup_{0\leq t\leq T}\|X^n_t\|^2\big], \E\big[\sup_{0\leq t\leq T} \|X^{*,n}_t\|^2\big]\leq M(1+ \|\bar x\|^2 +\|\bar y\|^2) e^{M T}, \text{ for $0\leq t\leq T$},\end{align}
for some constant $M>0$ possibly depending on $T$.}
\end{remark}
Let the process $X^*$ be the unique {strong solution (whose existence follows from same arguments as those used for~\eqref{eq-limit})} to 
\begin{align}\label{eq-limit-*}
\D   X^{*}_t= \bar b( X^{*}_t)\D t + \sigma( X^{*}_t) \D  W_t, \, X^{*}_0=\bar x
\end{align} 
\begin{theorem}\label{thm-main-weak}[Weak convergence estimate] {Fix $0<\theta<1$} and suppose Assumptions~\ref{assump-regularity} and~\ref{assump-stab} hold. Then for every $T>0$ and {$\phi\in \cC^{2,\theta}_b(\RR^p)$}, the following holds:
\begin{align}
\sup_{0\leq t\leq T}\Big|\E\Big[\phi( X^n_t)\Big]- \E\Big[\phi(X^*_t)\Big]\Big|\leq C{\|\phi\|_{2,\theta,\infty}}n^{-1},
\end{align}
for some {$C=C(T)>0$}. Here, the process $X^*$ is as defined in~\eqref{eq-limit-*}.
\end{theorem}
\begin{remark}The convergence rates obtained in Theorems~\ref{thm-main} and~\ref{thm-main-weak} are known to be optimal; see \cite[Section 2]{liu2010} and  we do not impose any regularity of coefficients other than  Lipschitz continuity and linear growth. As for boundedness, we only impose boundedness of $b$ {and $\sup_{z\in \RR^p}\|h(z,y_0)\|$, for some $y_0\in \RR^q$}. This is in contrast to most of the existing literature where the coefficients are often assumed to be differentiable more than once.
\end{remark}

\section{Key intermediate results }\label{sec-filter}
In this section, we collect all the intermediate results that are necessary for the proof of Theorem~\ref{thm-main}. These results include the Lipschitz continuity of  $\bar b$ (from Theorem~\ref{thm-main}) and the estimate of the ``closeness" between $\pi^n_t$ and $\pi^{*,X^n_t} $, for $0\leq t\leq T$, in terms of $n$. 
{\subsection{Continuity of $\overline b$} In this section, we obtain  the Lipschitz continuity of $\pi^{*,z}$ in terms of $z$ as an immediate consequence of existing results in the literature. This consequently also implies the Lipschitz continuity of $\bar b(\cdot)$. 
\begin{lemma}  \label{lem-lipschitz-inv-meas}Under Assumptions~\ref{assump-regularity} and~\ref{assump-stab}, there exists a constant $ L_{\text{inv}}>0$ such that 
\begin{align*}
\| \pi^{*,x_1}-\pi^{*,x_2}\|_{\text{TV}}\leq L_{\text{inv}}\|x_1-x_2\|,
\end{align*}
for every $x_1,x_2\in \RR^p$. Here, $\|\cdot\|_{\text{TV}}$ denotes the total variation norm. 
\end{lemma}
\begin{proof} Fix $x_1,x_2\in \RR^p$. From Assumption~\ref{assump-regularity} and  \cite[Theorem 2.6.16]{arapostathis2011}, $\pi^{*,x_1}$ admits a density denoted by $p_\infty(x_1, \cdot)$. Define, 
\begin{align*} \Phi (x_1,x_2,y)&\doteq \frac{\big(\eta(x_1,y)\eta(x_1,y)^\dagger\big)-\big(\eta(x_2,y)\eta(x_2,y)^\dagger\big)\nabla p_\infty (x_1,y)}{p_\infty(x_1,y)} \\
&\qquad+ h(x_1,y)-h(x_2,y) -\nabla_y \eta (x_1,y)+ \nabla_y\eta(x_2,y)\,.\end{align*}
From Lipschitz property of $h$, $\eta$ and $\nabla_y\eta$, and the boundedness of $\eta$, we have
\begin{align}\label{eq-phi-est} \| \Phi(x_1,x_2,y)\|\leq C\|x_1-x_2\| \Big(1+ \frac{\|\nabla p_\infty(x_1,y)\|}{p_\infty(x_1,y)}\Big), \text{ for some $C>0$}\,.\end{align}
Applying \cite[Theorem 2.2]{bogachev2018} for $m=1$, we have
\begin{align*}
\|\pi^{*,x_1}-\pi^{*,x_2}\|_{\text{TV}}&\leq C \int_{\RR^q} (1+\|y\|)| \Phi(x_1,x_2,y)| \pi^{*,x_1}(\D y)\\
&\leq C \Big(\int_{\RR^q} (1+\|y\|)^2\pi^{*,x_1}(\D y)\Big)^{\frac{1}{2}}\Big(\int_{\RR^q}  | \Phi(x_1,x_2,y)|^2 \pi^{*,x_1}(\D y)\Big)^{\frac{1}{2}}\\
&\leq \frac{C\beta_0}{\beta_1} \|x_1-x_2\|, 
\end{align*} 
for some $C>0$. To get the last line, we use~\eqref{eq-phi-est} and \cite[Remark 2.2(ii)]{bogachev2018}. This completes the proof with $L_{\text{inv}}\doteq C\beta_0\beta_1^{-1}$.
\end{proof}

\begin{lemma} \label{lem-lipschitz}Under Assumptions~\ref{assump-regularity} and~\ref{assump-stab}, $\bar b$ is Lipschitz \emph{i.e.,} for a constant $\text{Lip}(\bar b)>0$,
\begin{align*}
\|\bar b(x_1)-\bar b(x_2)\|\leq \text{Lip}(\bar b)\|x_1-x_2\|,\text{ for every $x_1,x_2\in \RR^p$.}
\end{align*}

\end{lemma}
\begin{proof} For every $x_1,x_2\in \RR^p$,  consider
\begin{align*}
&\Big\|\int_{\RR^q} b(x_1,y)\pi^{*,x_1}(\D y)- \int_{\RR^q} b(x_2,y)\pi^{*,x_2}(\D y)\Big\|\\
&\leq  \int_{\RR^n} \|b(x_1,y)- b(x_2,y)\|\pi^{*,x_1}(\D y)+\Big\|\int_{\RR^q} b(x_2,y)\pi^{*,x_1}(\D y)- \int_{\RR^n} b(x_2,y)\pi^{*,x_2}(\D y)\Big\|\\
&\leq \text{Lip}(b)\|x_1-x_2\| + \Big\|\int_{\RR^q} b(x_2,y)\pi^{*,x_1}(\D y)- \int_{\RR^q} b(x_2,y)\pi^{*,x_2}(\D y)\Big\|\\
&\leq \text{Lip}(b)\|x_1-x_2\| + \|b\|_\infty L_{\text{inv}}\|x_1-x_2\| \,.
\end{align*}
Setting $\text{Lip}(\bar b)\doteq \text{Lip}(b)+ \|b\|_\infty L_{\text{inv}} $, we have the result.
\end{proof}}
\subsection{{A weighted Lipschitz distance like function and an associated estimate}}
In this section, we prove a crucial lemma {(see Lemma~\ref{lem-est})} that provides an upper bound on a type of  weighted {Lipschitz distance like function   between a generic one-parameter family of measures $\bar \pi\doteq \{\bar \pi_t:0\leq t\leq T\}$ and the family of  unique invariant measures (of the ``frozen" diffusion parametrized by a path $\gamma:[0,T]\rightarrow\RR^p$) $\pi^{*,\gamma}\doteq \{\pi^{*,\gamma_t}:0\leq t\leq T\}$  in terms of $\int_0^T\bar \pi\big[\mathfrak{L}^{\gamma_t} [g]\big]\D t$}, where $g$ lies in an appropriate class of functions.  This lemma then helps us  analyze $ \pi^n_t(\D y)$ and its behavior as $n\to \infty$ {in Proposition~\ref{prop-est}}.  

From~\eqref{eq-KS}, it is clear that if terms other than ${n}\int_0^t \pi_{r}^n\big[\mathfrak L^{X^n_r} [f(X^n_r,\cdot)]\big]\D r$ are uniformly bounded in $n$ (in an appropriate sense), then $\int_0^t \pi_{r}^n\big[\mathfrak L^{X^n_r} [f(X^n_r,\cdot)]\big]\D r$ is small for every $0\leq  t\leq T$ for large $n$ and in particular, $\pi_{r}^n\big[\mathfrak L^{X^n_r} [f(X^n_r,\cdot)]\big]$ approaches zero as $n\to\infty$. This means that $\pi^n_t$ approaches $\pi^{*,X^n_t}$ as the condition:  $\pi^{*,z}\big[\mathfrak L^{z} [f]\big]=0$ (for any appropriately generic function $f$) completely characterizes $\pi^{*,z}$; see Remark~\ref{rem-inv-meas}. However, we are interested in quantifying this limiting behavior of $\pi^n_t$. This is the content of the rest of this section. We believe the estimates given below can be of  independent interest and therefore we obtain these results in a more general setup.

To begin with,  { we introduce a weighted supremum norm  for functions on $\RR^q$: for $f:\RR^q\rightarrow \RR$, 
$$\|f\|_*=\sup_{y\in \RR^q} \frac{|f(y)|}{1+ \cV(y)}\,.$$
With slight abuse of notation, if $g:\RR^q\rightarrow \RR^q$, we still use $\|g\|_*$ to denote $\sup_{y\in \RR^q} \frac{\|g(y)\|}{1+ \cV(y)}$. Now,  let $\cX$  be the set of continuously differentiable functions $f$ on $ \RR^q$, where two functions are considered identical  if their difference is a constant. 

For $z\in \RR^p$, $t\geq 0$ and $f\in \cX$, define, \begin{align}\label{eq-H} H^f_t( z,y)\doteq \E\Big[f(Y^{ z}_t)-\pi^{*,z}[f]\big| Y^{z}_0=y \Big]\,.\end{align}
\begin{remark}\label{rem-ergodic} For $k\geq 1$, define $\|f\|_{*,k}=\sup_{y\in \RR^q} \frac{|f(y)|}{1+ \cV_k(y)}\,.$ Under Assumptions~\ref{assump-regularity}, invoking \cite[Theorem 2.3 and Proposition 2.4]{lazic2021} we conclude that the process $Y^z$ is \emph{open-set irreducible and aperiodic} (see \cite[Pg. 350]{lazic2021} for their definitions). Then using \cite[Theorem 5.1 and Theorem 7.1]{tweedie1994topological}, we infer that all compacts sets are \emph{petite} (see \cite[Pg. 528]{meyn1993stability} for the definition). Now, from~\ref{assump-stab}, using  the well-known exponential ergodicity result \cite[Theorem 6.1]{meyn1993stability}, we know that 
\begin{align}\label{eq-exp-conv}
\sup_{\stackrel{f\in \cX}{\|f\|_{*,k}\leq 1}} \|H^f(z,\cdot)\|_{*,k} \leq \|f\|_{*,k}C_ke^{-\xi_k t},
\end{align}
for some constants (depending on $k$) $C_k>0$ and $\xi_k>0$ (independent of $z$). We suppress the dependence of $k$ on $(C_k,\xi_k)$ and simply write $(C,\xi)$.
\end{remark}

 Since we intend to estimate the `closeness' between two families of probability measures $\{\pi^n_t:0\leq t\leq T\}$ and $\{\pi^{*,X^n_t}:0\leq t\leq T\}$, we introduce a metric-like function on the family of probability measures as a whole. For a path $\gamma \in \cC([0,T];\RR^p)$ and a family of probability measures $\bar \pi\doteq \{\bar \pi_t: 0\leq t\leq T\}$, set $\pi^{*,\gamma}\doteq \{\pi^{*,\gamma_t}:0\leq t\leq T\}$ and define 
 \begin{align}\label{eq-norm-t}
 \rho (\bar \pi,\pi^{*,\gamma})\doteq \sup_{\stackrel{f\in \cX}{\|f\|_*,\|\nabla f\|_* \leq 1}}\Big| \int_0^T \big(\bar \pi_t[f]- \pi^{*,\gamma_t}[f]\big) \D t\Big|\,.
 \end{align}
 \begin{remark} We use the phrase `metric-like' above to note that we do not necessarily claim that $\rho$ in the above display is a metric.
 \end{remark}
} {To obtain the estimate of $\rho(\bar \pi, \pi^{*,\gamma})$ in terms of $\int_0^t \bar \pi_{r}\big[\mathfrak L^{\gamma_r} [f(X^n_r,\cdot)]\big]\D r$ for $f\in \cX$,  we introduce the well-known Poisson's equation below: for  $z\in \RR^p$ and $f\in \cX$, let $\bar V_f(z,\cdot)\in \cC^2(\RR^q)$ be the unique solution to 
\begin{align} \label{eq-PE-o}
\mathfrak L^z [\bar V_f(z,\cdot)](y)= f(y)-\pi^{*,z}[f]\,.
\end{align}

\begin{remark}Even though $\bar V_f(z,y)$ is twice differentiable in $y$, for every $z\in \RR^p$, to apply Proposition~\ref{prop-KS}, we require $\bar V_f(z,y)$ to be at least twice differentiable in $z$ as well. Unfortunately, we cannot ascertain this kind of regularity in $z$ by  considering~\eqref{eq-PE-o}. To circumvent this issue, we regularize the $\bar V_f$ in very particular manner so that the regularized version of $\bar V_f$  is more amenable to our requirements. \end{remark}

In light  of the above remark, we consider the following `perturbation' of~\eqref{eq-PE-o}:  for $\veps,\lambda>0$ (whose values will be chosen to be appropriately small later),
\begin{align}\label{eq-PE-p}
\veps\nabla^2_z V_f(z,y)+ \mathfrak L^z [V_f(z,\cdot)](y) -\lambda V_f(z,y)= f(y)-\pi^{*,z}[f]\,.
\end{align}
Here, $\nabla^2_z$ denotes the Laplacian with respect to $z$. From \cite[Lemma 3.5.4]{arapostathis2011}, there exists a unique function $V_f\in \cC^2(\RR^p\times \RR^q)$, for every $f\in \cX$ such that $V_f$ satisfies~\eqref{eq-PE-p}.  

From \cite[Section A.3]{arapostathis2011}, we know that $V_f$ can be expressed as follows: 
\begin{align}\label{eq-V-expression}
V_f(z,y)&= - \int_0^\infty  \int_{\RR^p} \frac{e^{-\lambda t}}{\sqrt{4 \veps^p t^p}} e^{-\frac{\|z-\bar z\|^2}{4\veps t}} H^f_t(\bar z,y) \D \bar z\D t\,.
\end{align}
Recall that 
\begin{align}\label{eq-gauss}
\frac{1}{\sqrt {4\veps^p t^p} }\int_{\RR^p}e^{-\frac{\|z-\bar z\|^2}{4\veps t}}  \D \bar z=1\,.
\end{align}
Before we proceed further, we prove important  estimates (in Lemmas~\ref{lem-property}--\ref{lem-est-gradV}) concerning $H_f$ and $V_f$, for  $f\in \cX$  in terms of $\|f\|_*$ and $\|\nabla f\|_*$. { In the following, we denote the semigroup associated with process $Y^z$ by $P^z_t(y,\D \bar y)$. In terms of $P^z_t$, we have $$H^f_t(z,y)=\int_{\RR^q} P^z_t(y,\D\bar y) \big(f(\bar y)-\pi^{*,z}[f]\big)\D \bar y\,.$$
In the next lemma, we collect important regularity properties of $H^f_t(z,y)$ 

\begin{lemma}\label{lem-property} Suppose Assumptions~\ref{assump-regularity} and~\ref{assump-stab} hold. For every $z\in \RR^p$, the semigroup $P^z_t(y,\D \bar y)$ admits a density $p_t(z,y,\bar y)$ that is differentiable in $y$. Moreover, $\nabla^{l_z}\nabla^{l_y}H^f_t(z,y)$ exists for $l_z,l_y=0,1$ and  $(z,y)\in \RR^p\times \RR^q$, and there exists a  constant $C>0$ (independent of $z$ and $t$) such that 
\begin{align*}
\|\nabla_z^{l_z}\nabla_y^{l_y}H^f_t(z,\cdot)\|_*\leq\begin{cases}C\max\big\{\| f\|_*,\|\nabla f\|_*\big\} &\text{ for $0<t\leq 1$},\\
Ce^{-\xi t}\max\big\{\| f\|_*,\|\nabla f\|_*\big\}&\text{ for $t>1$}\,.
\end{cases}
\end{align*}
\end{lemma}
The proof of this lemma is given in the appendix.
\begin{lemma}\label{lem-est-Vz} Suppose Assumptions~\ref{assump-regularity} and~\ref{assump-stab} hold. Then, there exists a constant $C>0$ (independent of $z$) such that 
for every  $f\in \cX$ and $z\in \RR^p$,  
\begin{align*} \max\Big\{\|V_f(z,\cdot)\|_*,\|\nabla_y V_f(z,\cdot)\|_*\Big\} \leq C \max\Big\{\frac{1}{\lambda+\xi}, \frac{1-e^{-\lambda}}{\lambda} + \frac{e^{-(\lambda+\xi) }}{\lambda+\xi}  \Big\}\max\big\{\|f\|_*,\|\nabla f\|_*\big\}\,.\end{align*} 
\end{lemma} 
\begin{proof} Fix $f\in \cX$. From~\eqref{eq-V-expression} and the definition of $H^f$ in~\eqref{eq-H}, we have 
\begin{align*}
\frac{|V_f(z,y)|}{1+\cV(y)}&\leq  \int_0^\infty  \int_{\RR^p} \frac{e^{-\lambda t}}{\sqrt{4 \veps^p t^p}} e^{-\frac{\|z-\bar z\|^2}{4\veps t}} \frac{|H^f(\bar z,y)|}{1+\cV(y)} \D \bar z\D t\,.\end{align*}
From the above display and Remark~\ref{rem-ergodic}, we obtain 
\begin{align}\label{eq-Vf-norm}
\frac{|V_f(z,y)|}{1+\cV(y)}&\leq C\|f\|_*\int_0^\infty\int_{\RR^p} \frac{e^{-(\lambda+\xi) t}}{\sqrt{4 \veps^p t^p}} e^{-\frac{\|z-\bar z\|^2}{4\veps t}}  \D \bar z\D t\leq \frac{C}{\lambda+\xi}  \|f\|_* \,. 
\end{align}
To get the  second inequality, we use~\eqref{eq-gauss}. Now consider 
\begin{align}\nonumber
\frac{\|\nabla_y V_f(z,y)\|}{1+\cV(y)}&=\Big| \int_0^1 \int_{\RR^p} \frac{e^{-\lambda t}}{\sqrt{4 \veps^p t^p}} e^{-\frac{\|z-\bar z\|^2}{4\veps t}} \frac{\nabla_y H^f_t(\bar z ,y)}{1+\cV(y)} \D \bar z\D t\Big| +\Big|\int_1^\infty\int_{\RR^p} \frac{e^{-\lambda t}}{\sqrt{4 \veps^p t^p}} e^{-\frac{\|z-\bar z\|^2}{4\veps t}} \frac{\nabla_y H^f_t(\bar z ,y)}{1+\cV(y)}\D \bar z\D t\Big|\\\nonumber
&\leq   \int_0^1 \int_{\RR^p} \frac{e^{-\lambda t}}{\sqrt{4 \veps^p t^p}} e^{-\frac{\|z-\bar z\|^2}{4\veps t}} \|\nabla_y H^f_t(\bar z,\cdot)\|_*\D \bar z\D t +\int_1^\infty\int_{\RR^p} \frac{e^{-\lambda t}}{\sqrt{4 \veps^p t^p}} e^{-\frac{\|z-\bar z\|^2}{4\veps t}} \|\nabla_y H^f_t(\bar z,\cdot)\|_* \D \bar z\D t\\\nonumber
&\leq C\max\{\|f\|_*,\|\nabla f\|_*\}\Big( \int_0^1 \int_{\RR^p} \frac{e^{-\lambda t}}{\sqrt{4 \veps^p t^p}} e^{-\frac{\|z-\bar z\|^2}{4\veps t}} \D \bar z\D t + \int_1^\infty \int_{\RR^p} \frac{e^{-(\lambda+\xi) t}}{\sqrt{4 \veps^p t^p}} e^{-\frac{\|z-\bar z\|^2}{4\veps t}} \D \bar z\D t \Big)\\\label{eq-gradV}
&\leq  C\max\{\|f\|_*,\|\nabla f\|_*\} \Big(\frac{1-e^{-\lambda}}{\lambda} + \frac{e^{-(\lambda+\xi) }}{\lambda+\xi} \Big)\,.
\end{align}
In the above, to get the second line, we use the definition of $\|\cdot\|_*$; to get the third line, we use Lemma~\ref{lem-property} (for $l_z=0$, $l_y=1$) and to get the fourth line, we again use~\eqref{eq-gauss}. Therefore, from the arbitrariness of $y$ above,~\eqref{eq-Vf-norm} and~\eqref{eq-gradV}, the proof is complete.
\end{proof}

\begin{lemma}\label{lem-est-gradV} For every $z\in \RR^p$ and $f\in \cX$, following  hold: for some constant $C>0$ (independent of $z$),
$$\max\big\{ \|\nabla^2_z\nabla_y V_f(z,y)\|_*, \|\nabla^2_z V_f(z,y)\|_*\big\}\leq  C\max\big\{\|f\|_*,\|\nabla f\|_*\big\} \big(\frac{1}{\sqrt \veps } +\frac{ e^{-(\lambda+\xi)}}{\sqrt \veps (\lambda+\xi)}\big)\,. $$

\end{lemma}
\begin{proof}
Fix $f\in \cX$ and recall~\eqref{eq-V-expression}:
\begin{align*}
\nabla_zV_f(z,y)&= - \int_0^\infty\int_{\RR^p} \frac{e^{-\lambda t}}{\sqrt{4 \veps^p t^p}}  e^{-\frac{\|z-\bar z\|^2}{4\veps t}} \nabla_zH^f_t(\bar z ,y) \D \bar z\D t\\
&= -\int_0^1 \int_{\RR^p} \frac{e^{-\lambda t}}{\sqrt{4 \veps^p t^p}} e^{-\frac{\|z-\bar z\|^2}{4\veps t}} \nabla_z H^f_t(\bar z ,y) \D \bar z\D t -\int_1^\infty\int_{\RR^p} \frac{e^{-\lambda t}}{\sqrt{4 \veps^p t^p}} e^{-\frac{\|z-\bar z\|^2}{4\veps t}} \nabla_zH^f_t(\bar z ,y)\D \bar z\D t\,.
\end{align*}
From here, we have
 \begin{align*}
\frac{ \big\|\nabla^2_z V(z,y)\big\|}{1+\cV(y)} &\leq  \Big\|\int_0^1\int_{\RR^p}  \frac{2e^{-\lambda t}\|z-\bar z\|}{\sqrt{4 \veps^{p+1} t^{p+1 }}}  e^{-\frac{\|z-\bar z\|^2}{4\veps t}}\frac{ \nabla_z H^f_t(\bar z,y)}{1+\cV(y)}\D \bar z\D t\Big\| \\
&\quad\quad+ \Big\|\int_1^\infty\int_{\RR^p} \frac{2 e^{-\lambda t}\|z-\bar z\|}{\sqrt{4 \veps^{p+1} t^{p+1}}}e^{-\frac{\|z-\bar z\|^2}{4\veps t}} \frac{\nabla_zH_t^f(\bar z,y)}{1+\cV(y)} \D \bar z\D t\Big\|\\
&\leq  \int_0^1\int_{\RR^p}  \frac{2e^{-\lambda t}\|z-\bar z\|}{\sqrt{4 \veps^{p+1} t^{p+1}}}  e^{-\frac{\|z-\bar z\|^2}{4\veps t}}\|\nabla_z H^f_t(\bar z,y)\|_* \D \bar z\D t \\
&\quad\quad+   \int_1^\infty\int_{\RR^p} \frac{2 e^{-\lambda t}\|z-\bar z\|}{\sqrt{4 \veps^{p+1} t^{p+1}}}e^{-\frac{\|z-\bar z\|^2}{4\veps t}}\|\nabla_z H^f_t(\bar z,y)\|_* \D \bar z\D t\\
&\leq C\max\big\{\|f\|_*,\|\nabla f\|_*\big\} \big(\frac{1}{\sqrt \veps } +\frac{ e^{-(\lambda+\xi)}}{\sqrt \veps (\lambda+\xi)}\big)\,.
 \end{align*}
To get the last line, we use Lemma~\ref{lem-property} for $l_z=1$ and $l_y=0$, and evaluate integrals in $t$. Now consider 
 \begin{align*}
 \frac{\|\nabla^2_z\nabla_yV_f(z,y)\|}{1+\cV(y)}&\leq \Big\|\int_0^1\int_{\RR^p}  \frac{2e^{-\lambda t}\|z-\bar z\|}{\sqrt{4 \veps^{p+1} t^{p+1 }}}  e^{-\frac{\|z-\bar z\|^2}{4\veps t}}\frac{ \nabla_z\nabla_y H^f_t(\bar z,y)}{1+\cV(y)}\D \bar z\D t\Big\| \\
&\quad\quad+ \Big\|\int_1^\infty\int_{\RR^p} \frac{2e^{-\lambda t}\|z-\bar z\|}{\sqrt{4 \veps^{p+1} t^{p+1}}}e^{-\frac{\|z-\bar z\|^2}{4\veps t}} \frac{\nabla_z\nabla_yH_t^f(\bar z,y)}{1+\cV(y)} \D \bar z\D t\Big\|\\
&\leq  \int_0^1\int_{\RR^p}  \frac{2e^{-\lambda t}\|z-\bar z\|}{\sqrt{4 \veps^{p+1} t^{p+1}}}  e^{-\frac{\|z-\bar z\|^2}{4\veps t}}\|\nabla_z\nabla_y H^f_t(\bar z,y)\|_* \D \bar z\D t \\
&\quad\quad+   \int_1^\infty\int_{\RR^p} \frac{2e^{-\lambda t}\|z-\bar z\|}{\sqrt{4 \veps^{p+1} t^{p+1}}}e^{-\frac{\|z-\bar z\|^2}{4\veps t}}\|\nabla_z\nabla_y H^f_t(\bar z,y)\|_* \D \bar z\D t\\
&\leq C\max\big\{\|f\|_*,\|\nabla f\|_*\big\} \big(\frac{1}{\sqrt \veps } +\frac{ e^{-(\lambda+\xi)}}{\sqrt \veps (\lambda+\xi)}\big)\,.
 \end{align*}
Again, to get the last line, we use Lemma~\ref{lem-property} for $l_z=1$ and $l_y=1$, and evaluate integrals in $t$. This completes the proof.
\end{proof}
}

\begin{lemma} \label{lem-est}Suppose Assumptions~\ref{assump-regularity} and~\ref{assump-stab} hold, and let $\bar \pi$ and $\pi^{*,\gamma}$ be as defined earlier.  Then, for $\veps, \lambda>0$ small enough, there exists $C=C(\veps,\lambda)>0$ such that 
\begin{align} \label{eq-contraction}
\rho(\bar\pi,\pi^{*,\gamma})\leq C\sup_{f\in \cX^*_1}\Big|\int_0^T  \bar \pi_t \big[\mathfrak{L}^{\gamma_t}[V_f(\gamma_t,\cdot)]\big]\D t \Big|\,.
\end{align}
Here, $\cX_1^*\doteq \{f\in \cX: \|f\|_*,\|\nabla f\|_*\leq 1\}$.
\end{lemma}
\begin{proof} For $f\in \cX$ such that $\|f\|_*,\|\nabla f\|_*\leq 1$,  consider
\begin{align*}
&\Big|\int_0^T \big(\bar \pi_t[f]- \pi^{*,\gamma_t}[f]\big) \D t\Big|\\
&= \Big|\int_0^T \big(\bar \pi_t\big[ \veps \nabla^2_zV_f(\gamma_t,\cdot)+ \mathfrak L^{\gamma_t} [V_f(\gamma_t,\cdot)]-\lambda V_f(\gamma_t,\cdot) \big]- \pi^{*,\gamma_t}\big[\veps \nabla^2_zV_f(\gamma_t,\cdot)+ \mathfrak L^{\gamma_t} [V_f(\gamma_t,\cdot)]-\lambda V_f(\gamma_t,\cdot)  \big]\big) \D t\Big|\\
&=  \Big|\int_0^T \bar \pi_t[\mathfrak L^{\gamma_t} [V_f(\gamma_t,\cdot)]]\D t + \int_0^T \big( \bar \pi_t\big[\veps \nabla^2_zV_f(\gamma_t,\cdot)- \lambda V_f(\gamma_t,\cdot) ]- \pi^{*,\gamma_t}\big[\veps \nabla^2_zV_f(\gamma_t,\cdot)-\lambda V_f(\gamma_t,\cdot)  \big]\big) \D t\Big|\\
&\leq \Big|\int_0^T \bar \pi_t[\mathfrak L^{\gamma_t} [V_f(\gamma_t,\cdot)]]\D t\Big| 
+ \lambda \Big|\int_0^T \big( \bar \pi_t\big[ V_f(\gamma_t,\cdot) ]- \pi^{*,\gamma_t}\big[ V_f(\gamma_t,\cdot)  \big]\big) \D t\Big|\\
&\qquad + \veps \Big|\int_0^T \big( \bar \pi_t\big[  \nabla^2_zV_f(\gamma_t,\cdot) ]- \pi^{*,\gamma_t}\big[\nabla^2_zV_f(\gamma_t,\cdot) \big]\big) \D t\Big|\\
&\leq \Big|\int_0^T \bar \pi_t[\mathfrak L^{\gamma_t} [V_f(\gamma_t,\cdot)]]\D t\Big| + C \lambda  \max\Big\{\frac{1}{\lambda+\xi}, \frac{1-e^{-\lambda}}{\lambda} + \frac{e^{-(\lambda+\xi) }}{\lambda+\xi}  \Big\}\max\big\{\|f\|_*,\|\nabla f\|_*\big\} \rho (\pi,\pi^{*,\gamma})\\
&\qquad + C \veps \max\big\{\|f\|_*,\|\nabla f\|_*\big\} \big(\frac{1}{\sqrt \veps } +\frac{ (e^{-(\lambda+\xi)})}{\sqrt \veps (\lambda+\xi)}\big) \rho (\pi,\pi^{*,\gamma})\,.
\end{align*} 
In the above, to get the first line, we use~\eqref{eq-PE-p}; to get the second line, we use the fact that  $\pi^{*,\gamma_t}\big[ \mathfrak L^{\gamma_t} [V_f(\gamma_t,\cdot)]\big]=0$; to get the last line, we use Lemmas~\ref{lem-est-Vz} and~\ref{lem-est-gradV}, along with the definition of $\rho$ in~\eqref{eq-norm-t}. From the choice of $f\in \cX$ and choosing $\veps,\lambda>0$ small enough and, we can ensure that $$K(\veps, \lambda)\doteq C \lambda  \max\Big\{\frac{1}{\lambda+\xi}, \frac{1-e^{-\lambda}}{\lambda} + \frac{e^{-(\lambda+\xi) }}{\lambda+\xi}  \Big\}+ C \veps  \big(\frac{1}{\sqrt \veps } +\frac{ (e^{-(\lambda+\xi)})}{\sqrt \veps (\lambda+\xi)}\big)<1\,. $$ From the arbitrariness of $f$ and definition of $\rho$, we have 
$$ \rho(\bar \pi,\pi^{*,\gamma})\leq \frac{1}{1-K(\veps,\lambda)}\sup_{f\in \cX^*_1}\Big|\int_0^T  \bar \pi_t \big[\mathfrak{L}^{\gamma_t}[V_f(\gamma_t,\cdot)]\big]\D t \Big|\,.$$ 
This completes the proof with $C(\veps,\lambda)\doteq \big(1-K(\veps,\lambda)\big)^{-1}$.
\end{proof}

In the next result, we estimate $\rho(\pi^n_t,\pi^{*,X^n_t})$, in terms of $n$ using Lemma~\ref{lem-est}. 
 \begin{proposition}\label{prop-est}
Suppose Assumptions~\ref{assump-regularity} and~\ref{assump-stab} hold and let $Z$ be a real valued non-negative random variable such that $\E[Z^2]<\infty$. Then for $\pi^n=\{\pi^n_r:0\leq r\leq T\}$ and $\pi^{*,X^n}=\{\pi^{*,X^n_r}:0\leq r\leq T\}$, we have
\begin{align}\label{eq-Z-0-est}
\E\Big[ Z\rho\big(\pi^n,\pi^{*,X^n}\big)\Big]\leq  C  n^{-1},
\end{align}
for some $C=C(T,\bar y,\E[Z^2])>0$. 
\end{proposition}
\begin{proof} 
For $f\in \cX^*_1$ (recall $\cX^*_1$ from Lemma~\ref{lem-est}), consider the corresponding $V_f$ defined \emph{via.}~\eqref{eq-PE-p}.  We now apply the Kushner-Stratanovich equation (which is~\eqref{eq-KS}) to function $ V_f(X^n_t,Y^n_t)$. Before we do this, we verify that $V_f$ satisfies the conditions of Proposition~\ref{prop-KS}. To begin with, we recall that  $ V_f (\cdot,\cdot)$ is twice differentiable in both the arguments. From Lemma~\ref{lem-est-Vz}, we have  
\begin{align}
|V_f(z,y)|\leq C \max\Big\{\frac{1}{\lambda+\xi}, \frac{1-e^{-\lambda}}{\lambda} + \frac{e^{-(\lambda+\xi) }}{\lambda+\xi}  \Big\}\max\big\{\|f\|_*,\|\nabla f\|_*\big\}\big(1+\cV(y)\big)\,.
\end{align}
Applying Remark~\ref{rem-k-moment} for $k\geq 2$, we can infer that 
\begin{align}\label{eq-sup-V} \sup_{0\leq t\leq T}\E\big[\big(1+\cV_k(Y^n_t)\big)^2\big]<\infty\,.\end{align}
Similarly, using Lemmas~\ref{lem-est-Vz} and~\ref{lem-est-gradV}  we have 
\begin{align*}
K\doteq \sup_{z\in \RR^p}&\max\big\{\|\nabla_zV_f(z,\cdot)\|_*,\|\nabla_y V_f(z,,\cdot)\|_*, \|\nabla_z^2 V_f(z,\cdot)\|_*\big\}<\infty\,.  \\
\implies \max&\big\{\|\nabla_zV_f(z,y)\|,\|\nabla_y V_f(z,y)\|, \|\nabla_z^2 V_f(z,\cdot)\|\big\}\leq K \big(1+\cV(y)\big)\,.
\end{align*}
From the definition of $\cV$, $\widehat {\mathfrak L}^{z,y}$ and  $\mathfrak L^{z}$ boundedness of $b$, the linear growth (uniform in $z$) of coefficients $\sigma(z)$, $h(z,y)$ and $\eta(z,y)$, we obtain
\begin{align*}
\sup_{z\in \RR^p}\big|\widehat {\mathfrak L}^{z,y}[V_f(z,y)\big|^2+ \big|\mathfrak L^{z} [V_f(z,y)]\big|^2\leq \bar K \big(1+ \| y\|^3\big)\,.
\end{align*}
Again, from Remark~\ref{rem-k-moment}, we get $\int_0^T \E\Big[\Big(\big|\widehat {\mathfrak L}^{X^n_s,Y^n_s}[f(X^n_s,Y^n_s)\big|^2+ \big|\mathfrak L^{X^n_s} [f(X^n_s,Y^n_s)]\big|^2\Big)\D s\Big]<\infty\,.$ This completes the verification of conditions of Proposition~\ref{prop-KS} \emph{viz.,}~\eqref{eq-KS-cond}.

 Now upon subsequent rearrangement of terms {in the Kushner-Stratanovich equation}, we get 
\begin{align*}
n \int_0^t \pi_{r}^n\big[\mathfrak L^{X^n_r} [ V_f(X^n_r,\cdot)]\big]\D r&= \pi^n_{t}[V_f(X^n_t,\cdot)]- V_f(\bar x,\bar y) -\int_0^t \pi_{r}^n\big[\widehat {\mathfrak L}^{X^n_r,\cdot}[V_f(X^n_r,\cdot)]\D r\\
 &\quad - \int_0^t \left(\pi^n_{r}[ V_f(X^n_r,\cdot)b^\dagger(X^n_r,\cdot)]-\pi_{r}^n[ V_f(X^n_r,\cdot)]\pi^n_{r}[b^\dagger(X^n_r,\cdot)]\right){\sigma^{-1}(X^n_r)}\D I^n_r, \quad \text{$\PP$--a.s.}
 \end{align*}
Recall $\cX^*_1$ from Lemma~\ref{lem-est}. From the above expression for $t=T$, we obtain the following: 
 \begin{align}\nonumber
 &n\E\Big[Z\sup_{f\in \cX^*_1} \Big| \int_0^T \pi_{r}^n\big[\mathfrak L^{X^n_r} [V_f(X^n_r,\cdot)]\big]\D r\Big|\Big]\\\nonumber
 &\leq \E[Z]\sup_{f\in \cX^*_1}\big|V_f(\bar x, \bar y)\big| +\int_0^T  \E\Big[Z\pi_{r}^n\big[\sup_{f\in \cX^*_1}| \widehat {\mathfrak L}^{X^n_r,\cdot}[V_f(X^n_r,\cdot)\big|] \Big]\D r +\E\Big[Z\sup_{f\in \cX^*_1} \big|\pi^n_{T}\big[ V_f(X^n_T,Y^n_T)\big|\big]\Big]  \\\label{eq-ks-est-1}
 &\quad+\E\Big[Z\sup_{f\in \cX^*_1}\big|\int_0^T \left(\pi^n_{r}[ V_f(X^n_r,\cdot)b^\dagger(X^n_r,\cdot)]-\pi_{r}^n[ V_f(X^n_r,\cdot)]\pi^n_{r}[b^\dagger(X^n_r,\cdot)]\right){\sigma^{-1}(X^n_r)}\D I^n_r\big| \Big]\\\nonumber
 &\doteq \sqrt{\E[Z^2]}\sup_{f\in \cX^*_1}\big|V_f(\bar x, \bar y)\big|+ J^n_1 +J^n_2+J^n_3\\\nonumber
 &\leq  C\big(1+\cV(\bar y)\big) + J^n_1+J^n_2+ J^n_3\,.
 \end{align}
Therefore, to complete the proof it suffices to establish the boundedness (uniform in $n$) of $J^n_1$, $J^n_2$ and $J^n_3$.

{ \bf (i) Boundedness of $J^n_1$:}  Consider
\begin{align*}
J^n_1&=\int_0^T  \E\Big[Z\sup_{f\in \cX^*_1}| \widehat {\mathfrak L}^{X^n_r,\cdot}[V_f(X^n_r,\cdot)\big| \Big]\D r\\
&\leq \int_0^T \Big( 2\E[Z^2] +2\E\Big[\sup_{f\in \cX^*_1}| \widehat {\mathfrak L}^{X^n_r,\cdot}[V_f(X^n_r,\cdot)\big|^2 \Big]\Big)\D r\\
&\leq   2\E[Z^2]T +2C \int_0^T\E\Big[ (1+\|Y^n_r\|^8) \Big]\D r\\
&=  2\E[Z^2]T +2CT+ 2C \int_0^T\E\Big[ \|Y^n_r\|^8 \Big]\D r\\
&\leq 2\E[Z^2]T +2CT+ 2C T\big( \cV_8(\bar y) + \beta_0(8) T\big) \,.
\end{align*}
In the above, to get the third line, we use Lemma~\ref{lem-est-gradV} and linear growth condition of $\sigma$ and $b$ and to get the last line, we use Remark~\ref{rem-k-moment} for $k=8$.

{ \bf (i) Boundedness of $J^n_2$:} From Lemma~\ref{lem-est-Vz},  we have
\begin{align*}
J^n_2=\E\Big[Z\sup_{f\in \cX^*_1} \big|\pi^n_{T}\big[ V_f(X^n_T,Y^n_T)\big|\big]\Big]
&\leq  \big(\E[Z^2]\big)^{\frac{1}{2}}\Big(\E\Big[\sup_{f\in \cX^*_1} \big|\pi^n_{T}\big[ \cV(Y^n_T)\big|^2 \Big]\Big)^{\frac{1}{2}}\\
&\leq \big(\E[Z^2]\big)^{\frac{1}{2}}\Big(\E\Big[\sup_{f\in \cX^*_1} \big|\pi^n_{T}\big[ \cV(Y^n_T)\big|^2 \Big]\Big)^{\frac{1}{2}} \\
&\leq \big(\E[Z^2]\big)^{\frac{1}{2}}\Big(\E\Big[\big(1+\cV(Y^n_T)\big)^2 \Big]\Big)^{\frac{1}{2}}\\
&\leq  \big(\E[Z^2]\big)^{\frac{1}{2}} \big( \cV_2(\bar y) + \beta_0(2) T\big) ^{\frac{1}{2}}\,.
\end{align*}
In the above, to get the third line, we use Lemma~\ref{lem-est-Vz} and to get the last line, we again use Remark~\ref{rem-k-moment} for $k=2$.

{\bf (ii) Boundedness of $J^n_3$:}  From~\ref{eq-sup-V}, Remark~\ref{rem-k-moment}  and the boundedness  of $b$,  we can conclude that  $\{M_t:0\leq t\leq T\}$ defined as
\begin{align}
M_t=M^{n,f}_t\doteq \int_0^T \left(\pi^n_{r}[ V_f(X^n_r,\cdot)b^\dagger(X^n_r,\cdot)]-\pi_{r}^n[ V_f(X^n_r,\cdot)]\pi^n_{r}[b^\dagger(X^n_r,\cdot)]\right)\sigma^{-1}(X^n_r)\D I^n_r
\end{align}
is indeed a square integrable martingale. With $\langle M\rangle_\cdot$ being the quadratic variation of $M$,  define a stopping time $\tau_t=\inf\{s:\langle M\rangle_t>s\}$.  
From \cite[Theorem V.1.6]{revuz1999continuous}, with Brownian motion $I^n$, we have the following representation of $M$: $M_t= I^n_{\langle M\rangle_t}\,.$
From the definition of $\langle M\rangle$ and the definition of $\|\cdot\|_*$, we have
\begin{align*}
\langle M\rangle _t &= \int_0^t \Big\|\pi^n_{r}[ V_f(X^n_r,\cdot)b^\dagger(X^n_r,\cdot)]-\pi_{r}^n[ V_f(X^n_r,\cdot)]\pi^n_{r}[b^\dagger(X^n_r,\cdot)]\sigma^{-1}(X^n_r) \Big\|^2 \D r\\
&\leq  \delta \|b\|^2_\infty \int_0^t \Big(\Big|\pi^n_{r}[ V_f(X^n_r,\cdot)]\Big|^2+\Big|\pi_{r}^n[ V_f(X^n_r,\cdot)] \Big|^2\Big) \D r\\
&\leq \delta \|b\|^2_\infty \|f\|^2_* \int_0^t \big(2+ 2\big|\pi^n_{r}[ \cV]\big|^2\big) \D r\\
&\doteq \|f\|_*^2\omega^n_t
\end{align*} 
From the above bound and the representation of $M$, we can infer the following: for $f\in \cX^*_1$,
\begin{align*}
|M_t| = |I^n_{\langle M\rangle _t}|\leq \sup_{0\leq s\leq \|f\|_*^2\omega^n_t}|I^n_s|\leq  \sup_{0\leq s\leq \omega^n_t}|I^n_s|\leq \sup_{0\leq s\leq \omega^n_T}|I^n_s|\,.
\end{align*}
Since the right hand side of the last inequality is independent of $f$, we have
\begin{align*}
J^n_3=\E\Big[Z\sup_{f\in \cX^*_1} |M_t|\Big] \leq \big(\E[Z^2]\big)^{\frac{1}{2}}\Big(\E\Big[Z\sup_{f\in \cX^*_1} |M_t|\Big] \Big)^{\frac{1}{2}}\leq \big(\E[Z^2]\big)^{\frac{1}{2}}\Big(\E\Big[\sup_{0\leq s\leq \omega^n_T}|I^n_s|^2\Big]\Big)^{\frac{1}{2}}\,.
\end{align*} From Doob's maximal inequality and the fact that $\sup_n\E\big[\omega^n_T\big]<\infty$, we obtain $\sup_nJ^n_3<\infty$\,.  This proves~\eqref{eq-Z-0-est}.
\end{proof}
\begin{remark} \label{rem-any-t}From the arguments in the proof of Proposition~\ref{prop-est},it is clear that  the estimate~\eqref{eq-Z-0-est}  remains unchanged when $T$ in the definition of  $\rho$ is replaced by any $0\leq t\leq T$.  Moreover, the constant $C>0$ in Proposition~\ref{prop-est} can be chosen uniformly in $0\leq t\leq T$.
\end{remark}
We now give an immediate corollary to the above proposition. This corollary has a critical use in the proof of Theorem~\ref{thm-main} given in the next section.
\begin{corollary}\label{cor-b-est} For $b$ in~\eqref{eq-X}, under Assumptions~\ref{assump-regularity} and~\ref{assump-stab}, the following hold: there exists a constant $C= C(\|b\|_*)$ such that  for $0\leq t\leq T$,
\begin{align}\label{eq-est-b-1} 
\E\Big[ \Big|\int_0^t\langle X^n_s-X^{*,n}_s, \Delta^n_s\rangle\D s\Big|\Big]\leq Cn^{-1}
\end{align}
Here, $
\Delta^n_t(X^n_t)\doteq \int_{\RR^q} b(X^n_t,y)\pi^n_t(\D y)  -\int_{\RR^q} b(X^n_t,y)\pi^{*, X^n_t}(\D y) $.
\end{corollary}
\begin{proof} Consider
\begin{align*} 
\E\Big[ \Big|\int_0^t\langle X^n_s-X^{*,n}_s, \Delta^n_s\rangle\D s\Big|\Big]&= \E\Big[ \Big|\int_0^t\|X^n_s-X^{*,n}_s\| \langle \frac{(X^n_s-X^{*,n}_s)}{\|X^n_s-X^{*,n}_s\|}, \Delta^n_s\rangle\D s\Big|\Big]\\
&\leq  \E\Big[ \sup_{0\leq s\leq T}\big(\|X^n_s\|+ \|X^{*,n}_s\|\big)\Big|\int_0^t\langle \frac{(X^n_s-X^{*,n}_s)}{\|X^n_s-X^{*,n}_s\|}, \Delta^n_s\rangle\D s\Big|\Big]\\
&\leq  \max\{\text{Lip}(b),\|b\|_*\} \E\Big[ \sup_{0\leq s\leq T}\big(\|X^n_s\|+ \|X^{*,n}_s\|\big) \rho (\pi^n, \pi^{*,X^n})\Big]\\ 
&\leq Cn^{-1}
\end{align*}
In the above, to get the third line, we use definitions of $\Delta^n_t$ and $\rho$ and to get the last line, we use Proposition~\ref{prop-est}. 
\end{proof}
}

\section{Proofs of our main results }\label{sec-proof} 
\subsection{Proof of Theorem~\ref{thm-main}}
{We first prove that~\eqref{eq-limit} has a unique strong solution. This follows immediately from the Lipschitz continuity of $\bar b$ which is established in Lemma~\ref{lem-lipschitz}.}

\begin{proof}[Completing the proof of Theorem~\ref{thm-main}] Recall that $X^n$ and $X^{*,n}$ satisfy 
~\eqref{eq-representation} and~\eqref{eq-limit}. Using this, $X^n$ can be further represented as 
\begin{align}\label{eq-representation-fin}
\D X^n_t=\int_{\RR^q} b(X^n_t,y)\pi^{*, X^n_t}(\D y)\D t + \Delta^n_t(X^n_t)\D t +\sigma(X^n_t) \D  I^n_t,  \,
\end{align} 
where  $\Delta^n_t(X^n_t)$ is as defined in Corollary~\ref{cor-b-est}. From here, the proof of Theorem~\ref{thm-main} reduces to the analysis of $X^n- X^{*,n}$ using  the representation~\eqref{eq-representation-fin} of $X^n$ and~\eqref{eq-limit}. 

We now apply It\^o's formula to $\|X^n_t-X^{*,n}_t\|^2$ (this is possible due to Remark~\ref{rem-sup-est})  to get
{\begin{align}\nonumber
\E\Big[\|X^n_t-X^{*,n}_t\|^2\Big] &\leq 2 \int_0^t \E\Big[  \|X^n_s-X^{*,n}_s\| \|\bar b(X^n_s)- \bar b( X^{*,n}_s)\|\Big]\D s +2 \E\Big[ \Big|\int_0^t \langle X^n_s-X^{*,n}_s, \Delta^n_s(X^n_s)\rangle \D s\Big|\Big]  \\\nonumber
&\quad+\int_0^t \E\Big[\|\sigma(X^n_s) -\sigma(X^{*,n}_s)\|^2\Big] \D s + 2\E\Big[\int_0^t \langle X^n_s-X^{*,n}_s,\big(\sigma(X^n_s) -\sigma(X^{*,n}_s) \big)\D I^n_s\rangle  \Big]\\\nonumber
&\leq \big(2 \text{Lip}(\bar b)+\text{Lip}(\sigma)\big)\int_0^t \E\Big[\|X^n_s-X^{*,n}_s\|^2\Big]\D s +2 \E\Big[ \Big|\int_0^t \langle X^n_s-X^{*,n}_s, \Delta^n_s(X^n_s)\rangle \D s\Big|\Big]\\
\label{eq-fin-1}
&\leq \big(2 \text{Lip}(\bar b)+\text{Lip}(\sigma)\big)\int_0^t \E\Big[\|X^n_s-X^{*,n}_s\|^2\Big]\D s +Cn^{-1}
\end{align}
In the second line, we use~\eqref{eq-moment-est} of Remark~\ref{rem-sup-est} and the linear growth property of $\sigma$ to conclude that the stochastic integral inside the expectation in the first line is in fact a square integrable martingale. Hence its expectation is zero. To get the second inequality, we use the Lipschitz property of  $\bar b$ (see Lemma~\ref{lem-lipschitz}) and $\sigma$ (see Assumption~\ref{assump-regularity}) and  to get the last inequality, we use Corollary~\ref{cor-b-est}.}    From Gr\"onwall's  inequality, we now have 
\begin{align}\label{eq-sup-out} \sup_{0\leq t\leq T}\E\Big[\|X^n_t- X^{*,n}_t\|^2\Big]\leq \widehat Cn^{-1}, \end{align}
for some $\widehat C=\widehat C(T)>0$.  This proves~\eqref{eq-thm-est-1}.

We now prove~\eqref{eq-thm-est-2} for $1<m<2$. This is done using a stopping time argument which is similar to that in \cite[Pg. 1529-1530]{champagnat2018}.  Define $A_t\doteq \sup_{0\leq s\leq t}\|X^n_s-X^{*,n}_s\|$ and for $l>0$, a stopping time $$\tau^n_l\doteq\inf\{t\geq 0:A_t\geq l\}\wedge (T+1)\,.$$ For $0\leq t\leq T$, repeating the above computation with $\|X^n_{t\wedge \tau^n_l}-X^{*,n}_{t\wedge \tau^n_l}\|^2$, we obtain 
\begin{align*}
\E\Big[\|X^n_{t\wedge \tau^n_l}-X^{*,n}_{t\wedge \tau^n_l}\|^2\Big] &\leq 2\E\Big[\int_0^{t\wedge \tau^n_l}   \|X^n_s-X^{*,n}_s\| \|\bar b(X^n_s)- \bar b(X^{*,n}_s)\|\D s\Big] +2 \E\Big[ \Big|\int_0^{t\wedge \tau^n_l} \langle X^n_s-X^{*,n}_s, \Delta^n_s(X^n_s)\rangle \D s\Big|\Big]  \\
&\quad+\E\Big[\int_0^{t\wedge \tau^n_l} \|\sigma(X^n_s) -\sigma(X^{*,n}_s)\|^2 \D s\Big]{+2\E\Big[\int_0^{t\wedge \tau^n_l} \langle X^n_s-X^{*,n}_s,\big(\sigma(X^n_s) -\sigma(X^{*,n}_s)\big)\D I^n_s\rangle   \Big]}\\
&\leq 2\E\Big[\int_0^{t}   \|X^n_s-X^{*,n}_s\| \|\bar b(X^n_s)- \bar b(X^{*,n}_s)\|\D s\Big] + 2 \E\Big[ \Big|\int_0^{t} \langle X^n_s-X^{*,n}_s, \Delta^n_s(X^n_s)\rangle \D s\Big|\Big] \\
&\quad+\E\Big[\int_0^{t} \|\sigma(X^n_s) -\sigma(X^{*,n}_s)\|^2 \D s\Big]\\
&\leq \frac{\big(2 \text{Lip}(\bar b)+\text{Lip}(\sigma)\big) \widehat C T}{n} +\frac{C}{n}\doteq  \frac{\widehat C'}{n}\,.
\end{align*}
In the above,  {to get the second inequality, we again use~\eqref{eq-moment-est} and the linear growth property of $\sigma$ to conclude that the stochastic integral inside the expectation in the first line is in fact a square integrable martingale; to get the third inequality, we again use the Lipschitz property of  $\bar b$ (see Lemma~\ref{lem-lipschitz}) and $\sigma$ (see Assumption~\ref{assump-regularity}) and  to get the last inequality, we use Corollary~\ref{cor-b-est}.}   From the above analysis, we can infer that for $t=T$, 
\begin{align*} \E\Big[\|X^n_{\tau^n_l}-X^{*,n}_{\tau^n_l}\|^2\Ind_{\{ \tau^n_l\leq T\}}\Big]+ \E\Big[\|X^n_{ T}-X^{*,n}_{T}\|^2\Ind_{\{\tau^n_l>T\}}\Big]&= \E\Big[\|X^n_{T\wedge \tau^n_l}-X^{*,n}_{T\wedge \tau^n_l}\|^2\Big]  \leq \widehat C'n^{-1}\\
\E\Big[\|X^n_{\tau^n_l}-X^{*,n}_{\tau^n_l}\|^2\Ind_{\{ \tau^n_l\leq T\}}\Big]&\leq \E\Big[\|X^n_{T\wedge \tau^n_l}-X^{*,n}_{T\wedge \tau^n_l}\|^2\Big]  \leq \widehat C'n^{-1}\\
l^2 \E\Big[\Ind_{\{ \tau^n_l\leq T\}}\Big]&\leq \E\Big[\|X^n_{T\wedge \tau^n_l}-X^{*,n}_{T\wedge \tau^n_l}\|^2\Big]  \leq \widehat C'n^{-1}
\,.\end{align*}
To get the last line, we use the fact that $\|X^n_{\tau^n_l}-X^{*,n}_{\tau^n_l}\|\geq l$ which follows from the definition of $\tau^n_l$ and the continuity of processes $X^n$ and $X^{*,n}$. To summarize, we obtain 
$$ \PP\big(\tau^n_l\leq T\big)\leq \frac{\widehat C'}{n l^2}\,.$$
Now it is easy to see that
\begin{align}\label{eq-fin-2}
\PP\Big(\sup_{t\in [0,T]}\|X^n_t-X^{*,n}_t\|\geq l\Big)=\PP\big(\tau^n_l\leq T\big)\leq  \frac{\widehat C'}{n l^2}
\end{align}
For $1<m<2$, consider 
\begin{align*}
\E\Big[\sup_{t\in [0,T]}\|X^n_t-X^{*,n}_t\|^m\Big]&= m\int_0^\infty x^{m-1} \PP\Big(\sup_{t\in [0,T]}\|X^n_t-X^{*,n}_t\|\geq x\Big) \D x\\
&\leq m\int_0^\infty x^{m-1} \Big( \frac{\widehat C'}{n x^2} \wedge 1\Big) \D x\\
&\leq  m\int_0^{\sqrt {\frac{\widehat C'}{n}}} x^{m-1} \D x +\frac{m\widehat C'}{n} \int_{\sqrt {\frac{\widehat C'}{n}}} ^\infty x^{m-3} \D x\\
&\leq \frac{2}{2-m} \Big(\frac{\widehat C'}{n}\Big)^\frac{m}{2}\,.
\end{align*}
In the above, to get the second line, we use the fact that probability is bounded from above by $1$ and~\eqref{eq-fin-2}. This completes the proof of Theorem~\ref{thm-main}. 
\end{proof}

\subsection{Proof of Theorem~\ref{thm-main-weak}} 
Fix { $0<\theta<1$ and $\phi\in \cC^{2,\theta}_b(\RR^p)$.  Also,} let $P^*_t(x, \D y)$ be the semigroup associated with process $X^*$; recall $X^*$ from~\eqref{eq-limit-*}.  Also, let $\mathfrak{L}^*$ be the infinitesimal generator of $X^*$. { Now consider $P^*_{t-r}[\phi](x)= \int_{\RR^p} \phi(y)P^*_{t-r}(x, \D y)$ which is differentiable in $r$ and twice differentiable in $x$, due to \cite[Theorem 2.1]{rubio2011existence}.  We then apply It\^o's formula to 
$P^*_{t-r}[\phi](X^n_r)$} to get 
\begin{align*} 
\E\big[ \phi(X^n_t)\big]&= \E\big[\phi(X^*_t)\big] + \E\Big[\int_0^t \big(\frac{\partial }{\partial r} P^*_{t-r}\big)[\phi](X^n_r) \D r+ \int_0^t \big(P^*_{t-r} [\mathfrak{L}^*\phi](X^n_r) + \Delta^n_r(X^n_r)\cdot\nabla P^*_{t-r}\big[\phi\big](X^n_r)\big)\D r\Big]\\
&\qquad + \E\big[ \int_0^t \langle \nabla P^*_{t-r}[\phi](X^n_r), \sigma(X^n_r)\D W_r\rangle\big]\\
&= \E\big[\phi(X^*_t)\big] + \E\Big[\int_0^t \big( \Delta^n_r(X^n_r)\cdot \nabla P^*_{t-r} [\phi](X^n_r)\big)\D r\Big]\,.
\end{align*}
{In the first line above, to get the term $P^*_{t-r} [\mathfrak{L}^*\phi](X^n_r)$, we use the fact that operators $P^*_{t-r}$ and $\mathfrak{L}^*$ commute - $P^*_{t}$ is the semigroup associated with $\mathfrak{L}^*$.}
To get the second equality, we use the fact that $$\big(\frac{\partial }{\partial r} P^*_{t-r}\big)[\phi](x)+P^*_{t-r} [\mathfrak{L}^*\phi](x)=0$$ { and also, the fact that the stochastic integral inside the expectation is a martingale. To see that it a martingale, we infer that $\sup_{0\leq t\leq T}\|\nabla P^*_{t-r}[\phi]\|_{\infty}$ from the argument given below and the linear growth of $\sigma$, in conjunction with~\eqref{eq-moment-est}.} From here, we have
\begin{align*}
\Big|\E\big[ \phi(X^n_t)\big]- \E\big[\phi(X^*_t)\big] \Big|&\leq \E\Big[\big|\int_0^t\Delta^n_r(X^n_r)\cdot\nabla P^*_{t-r}  [\phi](X^n_r)\D r\big|\Big]\\
&\leq \sup_{0\leq t\leq T}\|\nabla P^*_{t}[\phi]\|_\infty \max\{\text{Lip}(b),\|b\|_*\} \E\Big[\rho(\pi^n,\pi^{*,X^n})\Big]\\
&\leq \sup_{0\leq t\leq T}\|\nabla P^*_{t}[\phi]\|_\infty \max\{\text{Lip}(b),\|b\|_*\}  \bar Cn^{-1}\,.
\end{align*}
for some $\bar C=\bar C(T)>0$. To get the last inequality, we applied  Proposition~\ref{prop-est}.  {From \cite[Theorem 4.1]{lorenzi2000optimal}, we can infer the following: for a constant $K= K(T)>0$ (independent of $\phi$, but possibly depend on $T$) 
\begin{align*}\sup_{0\leq t\leq T}\| P^*_{t}[\phi]\|_{2,\theta,\infty} \leq K\|\phi\|_{2,\theta,\infty}\,. \end{align*}
In particular, this implies $\sup_{0\leq t\leq T}\|\nabla P^*_{t}[\phi]\|_{\infty} \leq K\|\phi\|_{2,\theta,\infty}$.

Therefore,   we get
$$ \sup_{0\leq t\leq T}\Big|\E\big[ \phi(X^n_t)\big]- \E\big[\phi(X^*_t)\big] \Big|\leq  \bar CK \|\phi\|_{2,\theta,\infty} n^{-1},$$}
for some $C=C(T)\doteq \bar C(T) K(T)>0$.  This proves the result.\qed

\section*{Acknowledgement}
We thank the referee for helpful comments that significantly improved the quality of the paper.
\appendix
{\section{Proof of Lemma~\ref{lem-property}} 
To begin with, the existence of differentiable density $p_t(z,y,\bar y)$ follows as a direct consequence of \cite[Theorem 1.2]{menozzi2021}. 
The rest of the  proof follows exactly along the same lines as the proof of \cite[Lemma 3.4]{rockner2019strongweakconvergenceaveraging} (for $l=0$) and the proof of \cite[Lemma 3.5]{rockner2019strongweakconvergenceaveraging} (for $l=1$; in Lemma 3.5 the case  $l=2$ is also considered, but we only restrict ourselves to arguments associated with $l=1$).  The arguments in the \cite[Lemmas 3.4 and 3.5]{rockner2019strongweakconvergenceaveraging} are built  on  appropriate estimates of $\|\nabla_y p_t\|$ and $\|\nabla_y^2 p_t\|$ (for all $t>0$) which are proved under certain assumptions on coefficients $\eta$ and $h$ \emph{viz.,}  boundedness and  H\"older continuity  in both $z,y$ and differentiability in $y$. In the following, we  give similar estimates of $\|\nabla_y p_t\|$ and $\|\nabla_y^2 p_t\|$ (for all $t>0$) under Assumptions~\ref{assump-regularity} and~\ref{assump-stab}. With these estimates at hand, we can then follow the arguments from aforementioned results in \cite{rockner2019strongweakconvergenceaveraging}.

To that end,  we define $\phi:\RR^p\times \RR^q\times \RR_+\rightarrow \RR^q$ as follows: for every $z\in \RR^p$,  $$  \frac{\D}{\D t} \phi(z,y,t)=h\big(z,\phi(z,y,t)\big)\, \text{ such that }\phi(z,y,0)=y\,. $$
Now from \cite[Theorem 1.2(ii)-(iii)]{menozzi2021}, there are constants $\alpha,C>0$ (independent of $z$) such that for $0<t\leq 1$,
\begin{align}\label{eq-grad-p}
\|\nabla_y p_t(z,y,\bar y)\|&\leq Ct^{-\frac{p+1}{2}} e^{-\alpha \frac{\|\phi(z,y,t)-\bar y\|^2}{t}}\\\label{eq-ggrad-p}
  \|\nabla^2_y p_t(z,y,\bar y)\|&\leq Ct^{-\frac{p+2}{2}} e^{-\alpha \frac{\|\phi(z,y,t)-\bar y\|^2}{t}}\,.
\end{align}

Below, we prove the following: there exists constant $C>0$ such that 
$$ \|\nabla_y p_t(z,y,\bar y)\| \leq Ce^{-\xi t}(1+\cV(y)), \text{ for $t\geq 1$}\,. $$
To do this, we adapt the technique used in \cite{pardoux2003} (in particular, Page 1178 of that paper). With $p_\infty(z,\bar y)$ being the density of the invariant measure $\pi^{*,z}$, we use the Chapman-Kolmogorov equation to infer that for $t=1$,
$$ \nabla_y p_\infty (z,\bar y)= 0= \int_{\RR^q} \nabla_yp_1(z,y,\hat  y)p_\infty(z,\bar y) \D \hat y\,.$$
Moreover, we can infer the following: for $t> 1$
\begin{align*}
\nabla_yp_t(z,y,\bar y)=\int_{\RR^q}\nabla_y p_{1}(z,y,\hat y)\big(p_{t-1}(z,\hat y,\bar y)- p_\infty(z,\bar y)\big)\D \hat y\,.
\end{align*}
From the above display and~\eqref{eq-grad-p}, we obtain
\begin{align*}
\frac{\big\|\nabla_yp_t(z,y,\bar y)\big\|}{1+\cV(y)}&\leq \frac{1}{1+\cV(y)}\int_{\RR^q}\|\nabla_y p_{1}(z,y,\hat y)\| \big|p_{t-1}(z,\hat y,\bar y)- p_\infty(z,\bar y)\big|\D \hat y\\
&\leq \frac{C}{1+\cV(y)}\int_{\RR^q}e^{-\alpha \|\phi(z,y,t)-\hat  y\|^2}\big|p_{t-1}(z,\hat y,\bar y)- p_\infty(z,\bar y)\big|\D \hat y\\
&\leq \frac{C}{1+\cV(y)} e^{-\xi t} \big(1+\cV(y)\big) \,.
\end{align*}
To obtain the third line, we use the argument in \cite[Page 1177]{pardoux2003} along with Remark~\ref{rem-ergodic}. 

Arguing similarly, we can also obtain 
$$ \|\nabla^2_y p_t(z,y,\bar y)\| \leq Ce^{-\xi t}(1+\cV(y)), \text{ for $t\geq 1$}\,. $$}

		\bibliographystyle{aomplain}	
	\bibliography{Two_scale_dynamics-R1}
	\end{document}